%% file: main.tex
\documentclass[pdflatex,sn-aps]{sn-jnl}

\usepackage{graphicx}%
\usepackage{multirow}%
\usepackage{amsmath,amssymb,amsfonts}%
\usepackage{amsthm}%
\usepackage{mathrsfs}%
\usepackage[title]{appendix}%
\usepackage{xcolor}%
\usepackage{textcomp}%
\usepackage{manyfoot}%
\usepackage{booktabs}%
\usepackage{algorithm}%
\usepackage{algorithmic}%
\usepackage{listings}%

\usepackage[english]{babel}
\usepackage{float}
\usepackage{graphicx}
\usepackage{subcaption}
\usepackage{enumerate}
\usepackage{booktabs}
\usepackage{dsfont}
\usepackage{comment}
\usepackage{bm}
\usepackage{microtype}

\usepackage{todonotes}
\usepackage{tablefootnote}

\usepackage{tikz}
\usepackage{pgfplots}
\pgfplotsset{minor tick style={draw=none}}
\usepackage{caption}
\usepackage{float}

\allowdisplaybreaks[3]

\newtheorem{lemma}{Lemma}[section]
\newtheorem{definition}{Definition}[section]

\newtheorem{theorem}{Theorem}[section]

\newtheorem{proposition}{Proposition}[section]
\newtheorem{remark}{Remark}[section]

\newtheorem{assumption}{Assumption}[section]

\numberwithin{equation}{section}

\def\R{{\mathbb{R}}}
\def\B{{\mathbb{B}}}

\def\E{{\mathbb{E}}}

\def\argmin{\mathop{\rm arg\,min}}
\def\Argmin{\mathop{\rm Arg\,min}}

\def\C1loc{{C^{1,1}_{\rm loc}(\R^d)}}
\def\C1{{C^{1,1}(\R^d)}}
\def\suppf{{\frak s}}

\definecolor{cb-black}      {RGB}{  0,   0,   0}
\definecolor{cb-blue-green} {RGB}{  0,  073,  073}
\definecolor{cb-green-sea}  {RGB}{  0, 146, 146}
\definecolor{cb-rose}       {RGB}{255, 109, 182}
\definecolor{cb-salmon-pink}{RGB}{255, 182, 119}
\definecolor{cb-purple}     {RGB}{ 73,   0, 146}
\definecolor{cb-blue}       {RGB}{ 0, 109, 219}
\definecolor{cb-lilac}      {RGB}{182, 109, 255}
\definecolor{cb-blue-sky}   {RGB}{109, 182, 255}
\definecolor{cb-blue-light} {RGB}{182, 219, 255}
\definecolor{cb-burgundy}   {RGB}{146,   0,   0}
\definecolor{cb-brown}      {RGB}{146,  73,   0}
\definecolor{cb-clay}       {RGB}{219, 209,   0}
\definecolor{cb-green-lime} {RGB}{ 36, 255,  36}
\definecolor{cb-yellow}     {RGB}{255, 255, 109}

\begin{document}

\title[A Gaussian smoothing-based zeroth-order method]{A Gaussian smoothing-based zeroth-order method for Goldstein second-order stationarity}


\author[1]{Ming Lei}\email{leim@cuit.edu.cn}

\author[2]{Ting Kei Pong}\email{tk.pong@polyu.edu.hk}

\author[2]{Man-Chung Yue}\email{manchung.yue@polyu.edu.hk}

\author*[3]{Liaoyuan Zeng}\email{zengly@zjut.edu.cn}

\author[2]{Hao Zhang}\email{haaoo.zhang@connect.polyu.hk}

\affil[1]{Chengdu University of Information Technology, Chengdu, People's Republic of China}

\affil[2]{The Hong Kong Polytechnic University, Hong Kong, People’s Republic of China}

\affil[3]{Zhejiang University of Technology, Hangzhou, People’s Republic of China}

\abstract{We introduce a new generalized Hessian, called the Goldstein second-order $\delta$-subdifferential, and an associated notion of $(\epsilon_1,\epsilon_2,\delta)$-second-order stationary point for continuously differentiable functions with locally Lipschitz gradients. We propose a zeroth-order algorithm based on cubic regularization and Gaussian smoothing with homotopy to find such approximate second-order stationary points for Lipschitz differentiable functions, and derive the iteration complexity under a mild coercivity-type assumption on the objective function.}




\maketitle

\input{Second_order_GS_Vsigma.tex}

\noindent {\bf Acknowledgements.} The works of the second and fifth authors are supported partly
by a Research Scheme of the Research Grants Council of Hong Kong SAR, China (project T22-504/21R) and a PolyU internal grant 4-ZZV9. Man-Chung Yue is partially supported by the Hong Kong Research Grants Council under the GRF project~17309423. Liaoyuan Zeng was supported in part by Zhejiang Provincial Natural Science Foundation of China under Grant No.
LMS25A010004 and the National Natural Science Foundation of China under
Grant No.
12201389.

\medskip
\noindent {\bf Data availability.} There are no data associated with this paper.

\medskip
\noindent {\bf Competing interests.} The second author is a member of the editorial board of this journal.

\end{document}

%% file: Second_order_GS_Vsigma.tex
\section{Introduction}		
Zeroth-order optimization deals with optimization problems in which the objective function can only be queried through a zeroth-order oracle that returns its value at a prescribed point. Such problems are encountered in many contemporary applications, including the training and fine-tuning of neural networks~\cite{MalladiMeZO23, ChenZOO17}, simulation-based optimization~\cite{ConnSheinbergVicente09,AmaranSimOpt16}, and policy optimization in reinforcement learning~\cite{SalimansES17}. For a more comprehensive treatment of classical and recent developments in this area, we refer the readers to the monograph \cite{ConnSheinbergVicente09} and the survey \cite{LarsonMenickellyiWild19}.

Zeroth-order methods have well-developed complexity guarantees for obtaining first-order stationary points when the objective function is smooth \cite{ConnSheinbergVicente09,LarsonMenickellyiWild19, JongeneelYueKuhn24}. However, many contemporary applications give rise to nonsmooth objectives, including the training of neural networks with non-differentiable activation functions~\cite{MalladiMeZO23}, $\ell_1$-regularized formulations arising in sparse learning and signal processing~\cite{HuangGuHuoChenHuang19}, and $H_\infty$ control problems~\cite{GuoKeivanDullerudSeilerHu23}. This has motivated a growing line of work on zeroth-order methods for nonsmooth, locally Lipschitz objectives. Nesterov and Spokoiny~\cite{NesterovGS} developed the first zeroth-order method with complexity guarantee for nonsmooth Lipschitz objectives based on Gaussian smoothing (GS), and derived complexity for obtaining a point at which the GS function has small gradient. Since the GS gradient does not directly correspond to a generalized derivative of the original objective, subsequent works \cite{LinZHengJordan22,Pong2024,XiaLinDeng25} have instead focused on deriving complexity guarantees with respect to the notion of Goldstein $(\epsilon,\delta)$-stationarity, which is defined in terms of the Goldstein $\delta$-subdifferential \cite{Goldstein1977} of the objective itself.

Most existing zeroth-order methods for nonconvex optimization come with only first-order guarantees. First-order stationarity is, however, often insufficient: for many smooth nonconvex problems arising in modern applications, such as phase retrieval \cite{LiuYueSo17,SunQuWright18}, low-rank matrix recovery \cite{Bhojanapalli16,GeLeeMa16,YueZhouSo19} and dictionary learning \cite{SunQuWright17}, it has been shown that all second-order stationary points are global minimizers, while first-order stationary points need not be. Very recently, analogous landscape results have been established for some nonsmooth nonconvex problems as well \cite{GuanSo24,McRae25}. These developments have motivated a rich line of work on first- and second-order methods for finding approximate second-order stationary points of smooth \cite{NesPol06,CartisGouldToint12,JinNetrapalliGeKakadeJordan17} and nonsmooth \cite{DavisDiazDru22,HuangZhu23,DavisDruJiang26} nonconvex objectives; here, we need to point out that the notions of approximate second-order stationarity adopted in \cite{DavisDiazDru22,HuangZhu23,DavisDruJiang26} and the corresponding first- and second-order algorithms rely on additional structural assumptions such as weak convexity, partial smoothness or composite structure.
Complexity guarantees for finding approximate second-order stationary points have also been obtained for zeroth-order methods \cite{Bala2022,VlatakisGkaragkounisFlokasPiliouras19,RenTangLi23}, but only under the assumption that the objective is twice continuously differentiable with Lipschitz Hessian. To the best of our knowledge, no analogous result is available for nonsmooth objectives, and not even for functions that are differentiable but not twice differentiable. This raises a natural question: is there a meaningful notion of approximate second-order stationarity for nonsmooth (or, less ambitiously, non-twice-differentiable) nonconvex problems, and can one develop a zeroth-order method with a complexity guarantee for finding such points?

Two challenges need to be addressed when answering this question. First, in the absence of additional structures, formulating a meaningful notion of approximate second-order stationarity in the absence of twice differentiability is not obvious. The classical condition $\nabla^2 f(\bar x) \succeq 0$ has to be replaced by a suitable analog involving a generalized Hessian. Moreover, even with such an analog at hand, checking the second-order condition is plausibly intractable, as evidenced by hardness results in the related setting of smooth constrained optimization~\cite{MurtyKabadi87,NouiehedLeeRazaviyayn18}. One needs a tractable relaxation analogous to how the Goldstein $\delta$-subdifferential serves as a tractable surrogate for the Clarke subdifferential in the first-order setting. Second, the class of nonsmooth functions to be considered must be chosen so that the resulting notion of second-order stationarity is well-behaved. Recall that locally Lipschitz functions are the natural class for the first-order Goldstein framework, because their Clarke subderivatives are sublinear and locally bounded above, ensuring that the Clarke subdifferential is a nonempty, compact, convex set. It is therefore natural to seek a class of functions whose generalized second-order derivative possesses the same structural properties. By \cite[Theorem~1]{Ioffe02}, the generalized second-order derivative of $f$ (in the sense of \cite[Definition~1.1]{Cominetti90}; see also \cite[Definition~1]{Ioffe02}) is bi-sublinear and locally bounded above precisely when $f$ is continuously differentiable with locally Lipschitz gradient, i.e., $f \in C^{1,1}(\R^d)$. This characterization singles out $C^{1,1}(\R^d)$ as the first natural class of (differentiable but not necessarily twice differentiable) functions for us to develop a Goldstein-style notion of second-order stationarity. 

In view of the above discussion, in this paper we introduce a new notion of approximate second-order stationary point for functions in $C^{1,1}(\R^d)$, and develop and analyze a zeroth-order algorithm for finding such a point. Our contributions are summarized as follows.

\begin{itemize}
\item We introduce a new notion of generalized Hessian, which we call the Goldstein second-order $\delta$-subdifferential and denote by $\partial^2_{G,\delta} f$, and the associated notion of $(\epsilon_1, \epsilon_2, \delta)$-second-order stationary point. The construction is a natural second-order analog of the Goldstein $\delta$-subdifferential and the Goldstein $(\epsilon, \delta)$-stationary point, and gives rise to a tractable surrogate for the classical second-order optimality condition on $C^{1,1}(\R^d)$.

\item We establish a connection between the Hessian of the Gaussian smoothing $f_\sigma$ of $f$ and the Goldstein second-order $\delta$-subdifferential of $f$ when $f$ is Lipschitz differentiable. Specifically, we show that $\nabla^2 f_\sigma$ is well-defined whenever $\nabla f$ is Lipschitz, and that for sufficiently small $\sigma$ (with an explicit upper bound depending on $\epsilon$ and $\delta$), the minimum eigenvalue of $\nabla^2 f_\sigma$ provides a meaningful lower bound for $\inf_{\|h\|=1}\suppf(h,\partial^2_{G,\delta}f(\cdot)h)$; here, ${\frak s}(\cdot,\Omega)$ is the support function of the set $\Omega$. We should note that our representation formula for $\nabla^2 f_\sigma$ only requires $f$ to have a Lipschitz continuous gradient, weakening the Lipschitz Hessian assumption used in \cite{Bala2022}.

\item We propose a zeroth-order algorithm based on Gaussian smoothing and cubic regularization for finding an $(\epsilon_1, \epsilon_2, \delta)$-second-order stationary point of a Lipschitz differentiable function $f$. Our algorithm features {\em controlled increasing} sample sizes for approximating $\nabla f_\sigma$ and $\nabla^2 f_\sigma$ at each iteration, and a vanishing sequence $\{\sigma_k\}$ of smoothing parameters, making our approach an instance of homotopy method. Unlike the cubic-regularization-based zeroth-order method in \cite{Bala2022}, our algorithm does not require users to specify a tolerance in advance.
We analyze the iteration complexity of our algorithm, and derive the complexity guarantee upon assuming a mild coercivity-type condition on $f$.
\end{itemize}

\section{Notation and preliminaries}
Throughout this paper, we let $\mathbb{R}^{d}$ denote the Euclidean space of dimension $d$, and $\mathbb{N}$ denote the set of all natural numbers. For an $x\in \R^d$, we let $\|x\|$ denote its Euclidean norm, and use $\mathbb{B}(x,r)$ to denote the closed ball with center $x$ and radius $r\geq 0$. We use $\B_r$ to denote $\mathbb{B}(0,r)$. We let $I$ denote the $d\times d$ identity matrix, and its $j$-th column is denoted by $e_j$. Finally, for a matrix $M\in \mathbb{R}^{d\times d}$, we let $\|M\|$ denote its operator norm and $\|M\|_F := \sqrt{{\rm trace}(M^TM)}$.

We let $\C1$ denote the class of functions that are continuously differentiable with locally Lipschitz gradients.
For an $f\in \C1$, its generalized Hessian at $x\in \mathbb{R}^d$ is defined as
\begin{align}\label{def_Hess}
 \!\!\!\!\partial^2_Hf(x) \!:=\! {\rm conv}\{M: \exists x^{k}\!\rightarrow\! x \!\text{ s.t. $f$ is twice differentiable at }x^{k} \,\&\, \nabla^{2}f(x^{k})\!\rightarrow\! M  \}.
\end{align}
One can show that  the set $\partial_{H}^{2}f(x)$ is compact, and its elements are symmetric matrices thanks to \cite[Theorem~8.12.2]{JDieudonne69}. Moreover, from \cite[Section 1]{Pa96} (see also \cite[Section 2]{Hu84}), we know that
\begin{equation}\label{Equality}
\partial_{H}^{2}f(x) = \partial(\nabla f)(x) \ \ \ \forall x\in \mathbb{R}^d,
\end{equation}
where $\partial(\nabla f)(x)$ is the Clarke Jacobian of $\nabla f$ at $x$. In particular, we see from \cite[Proposition~2.6.2]{Clarke1983} that $x\mapsto \partial_H^2f(x)$ is upper semicontinuous in the sense of \cite[Proposition~2.6.2(c)]{Clarke1983}, and is locally uniformly bounded.
We say that $\bar x$ is a (first-order) stationary point of $f$ if $\nabla f(\bar x) = 0$, and we say that $\bar x$ is a second-order stationary point if
\begin{equation}\label{2nd-order_stationary}
\nabla f(\bar x) = 0\ \ {\rm and}\ \ \inf_{\|h\|=1}\suppf(h,\partial^2_Hf(\bar x)h) \ge 0;
\end{equation}
here, $\partial^2_Hf(\bar x)h := \{Mh:\; M\in \partial^2_Hf(\bar x)\}$ and ${\frak s}(\cdot,\Omega)$ is the support function of the set $\Omega$. It can be shown that any local minimizer of $f$ is a second-order stationary point; see \cite[Theorem 3.1]{Hu84}.

We next recall the definition of Gaussian smoothing (GS).
	\begin{definition}[{\cite[Section 2]{NesterovGS}}]\label{def_GSfunction}
		Let $\sigma > 0$. For a measurable function $f:\R^d\to \R$, we define its GS function as
		\[
		f_\sigma(x) = \mathbb{E}_{u\sim\mathcal{N}(0,I)}[f(x + \sigma u)],
		\]
		where $\mathcal{N}(0,I)$ denotes the $d$-dimensional standard Gaussian distribution. 	
	\end{definition}

The expectation in Definition~\ref{def_GSfunction} may not be well defined for general measurable function, but it is routine to show that such an expectation is well defined when $f$ has Lipschitz gradient. More importantly, the GS of a Lipschitz differentiable function is continuously differentiable and admits a ``derivative-free" expectation formula, as we recall in the next lemma.
\begin{lemma}[{\cite[Eq.~(21)]{NesterovGS}},{\cite[Theorem~3.3]{Pong2024}}]\label{lem053101}
Let $f$ be Lipschitz differentiable and $\sigma > 0$. Then its GS function $f_\sigma$ given in Definition \ref{def_GSfunction} is well defined. Moreover, the gradient of $f_\sigma$ is given by
	\begin{equation}\label{GSgradient}
		\nabla f_\sigma(x) = \frac1\sigma\mathbb{E}_{u\sim\mathcal{N}(0,I)}[f(x + \sigma u)u]
	\end{equation}
and is well defined and continuous.
\end{lemma}

We end this section with two lemmas concerning random variables that will be used in our complexity analysis in Section~\ref{sec:alg}. The next lemma generalizes \cite[Lemma~4.4]{Pong2024}, which corresponds to the case when $\theta = 2$.
\begin{lemma}\label{lem082701}
Let $\theta>0$, $x$ be a $d$-dimensional random vector and $g$ be a nonnegative lower semicontinuous function. Assume that
\begin{align}\label{conditionss}
 \mathbb{E}_{x}[\|x\|^{\theta}]\le\beta_{c}\ \ \text{and}\ \ \mathbb{E}_{x}\left[\frac{g(x)}{1+\|x\|^{n}}\right]\le\alpha_{c}
\end{align}
for some integer $n\ge1$ and nonnegative numbers $\alpha_{c}$ and $\beta_{c}$. Then
\begin{align*}
 \mathbb{E}_{x}\left[g(x)^{\frac{1}{2\lceil n/\theta\rceil}}\right]\le\left(1+\sqrt{\beta_{c}}\right)(2\alpha_{c})^{\frac{1}{2\lceil n/\theta\rceil}}.
\end{align*}
\end{lemma}
\begin{proof}
Notice that
\begin{align}\label{eq082601}
&\left(\mathbb{E}_{x}\Bigg[\Bigg(\frac{g(x)^{\frac{1}{2\lceil n/\theta\rceil}}}{1+\|x\|^{\theta/2}}\Bigg)^{2}\Bigg]\right)^{1/2}
\overset{\rm (a)}\le\Bigg(\mathbb{E}_{x}\Bigg[\Bigg(\frac{g(x)^{\frac{1}{2\lceil n/\theta\rceil}}}{1+\|x\|^{\theta/2}}\Bigg)^{2\lceil n/\theta\rceil}\Bigg]\Bigg)^\frac1{2\lceil n/\theta\rceil}\notag\\
& \ \ \ \ = \Bigg(\mathbb{E}_{x}\Bigg[\frac{g(x)}{(1+\|x\|^{\theta/2})^{2\lceil n/\theta\rceil}}\Bigg]\Bigg)^\frac1{2\lceil n/\theta\rceil}
\!\!\overset{\rm (b)}\le\! \Bigg(2 \mathbb{E}_{x}\Bigg[\frac{g(x)}{1+\|x\|^{n}}\Bigg]\Bigg)^\frac1{2\lceil n/\theta\rceil}
\!\!\!\le\!(2\alpha_{c})^{\frac{1}{2\lceil n/\theta\rceil}},
\end{align}
where (a) follows from Jensen's inequality,
(b) holds since $1+\|x\|^{n}\le2(1+\|x\|^{\theta\lceil n/\theta\rceil})\le2(1+\|x\|^{\theta/2})^{2\lceil n/\theta\rceil}$, and the last inequality follows from \eqref{conditionss}.
On the other hand, note that $\mathbb{E}_{x}[\|x\|^{\theta}]\le\beta_{c}$ implies $\mathbb{E}_{x}[\|x\|^{\theta/2}]\le\sqrt{\beta_{c}}$ and hence
$\mathbb{E}_{x}[(1+\|x\|^{\theta/2})^{2}]\le(1+\sqrt{\beta_{c}})^{2}$.
Combining this with \eqref{eq082601} gives
\begin{align*}
 (1+\sqrt{\beta_{c}})^2(2\alpha_{c})^{\frac{1}{\lceil n/\theta\rceil}}
 \!\ge\!
\mathbb{E}_{x}[(1+\|x\|^{\theta/2})^{2}]\cdot\mathbb{E}_{x}\Bigg[\Bigg(\frac{g(x)^{\frac{1}{2\lceil n/\theta\rceil}}}{1+\|x\|^{\theta/2}}\Bigg)^{2}\Bigg]
\!\ge\!
 \Bigg(\mathbb{E}_{x}\bigg[g(x)^{\frac{1}{2\lceil n/\theta\rceil}}\bigg]\Bigg)^2.
\end{align*}
\end{proof}

The next lemma is  a version of Rosenthal's inequality, which was also used in \cite{Bala2022} to analyze the complexity of their zeroth-order methods.
\begin{lemma}[{\cite[Theorem 2.1]{Rio2009},\cite[Theorem 8]{Bala2022}} ]\label{4moment}
Let $X_{j}, j=1,\ldots,n$ be independent mean-zero random variables. Then, we have $\mathbb{E}[|\sum_{j=1}^{n}X_{j}|^{4}]\leq[3\sum_{j=1}^{n}(\mathbb{E}[|X_{j}|^{4}])^{1/2}]^{2}$.
\end{lemma}

\section{Goldstein second-order $\delta$-subdifferential}

For any $f \in \C1$, any local minimizer $\bar x$ is a second-order stationary point in the sense of \eqref{2nd-order_stationary}. It is thus tempting to develop algorithms for finding points that approximately satisfy \eqref{2nd-order_stationary}, e.g., finding an $\bar x$ such that for some $\epsilon_1 > 0$ and $\epsilon_2>0$,
\begin{equation}\label{notion_2nd-order_stationary}
\|\nabla f(\bar x)\|\le \epsilon_1\ \ {\rm and}\ \ \inf_{\|h\|=1}\suppf(h,\partial^2_Hf(\bar x)h) \ge -\epsilon_2.
\end{equation}
However, given that checking the second-order stationarity for smooth constrained optimization problems is known to be intractable~\cite{MurtyKabadi87,NouiehedLeeRazaviyayn18}, it is unclear whether one can find such an $\bar x$ satisfying \eqref{notion_2nd-order_stationary} in polynomial time. 
Motivated by the success of Goldstein subdifferential as a tractable surrogate in the first-order setting, we consider the following ``relaxation" of the set $\partial^2_Hf(x)$ for defining a notion of approximate second-order stationarity. Our construction involves taking convex hull and considering sets $\partial^2_Hf(y)$ for $y$ close to $x$, which is analogous to the construction of the (first-order) Goldstein $\delta$-subdifferential.

\begin{definition}\label{2nd subdifferential}
 Let $f\in \C1$ and $\delta > 0$. The Goldstein second-order $\delta$-subdifferential of $f$ at $x\in \R^d$ is defined as
  \begin{align*}
  \partial_{G,\delta}^{2}f(x):=\mathrm{conv}\left(\bigcup_{y\in\mathbb{B}(x,\delta)}\partial_{H}^{2}f(y)\right).
  \end{align*}
\end{definition}

\begin{remark}\label{remark3}
In view of \eqref{def_Hess}, \eqref{Equality}, \cite[Proposition~2.6.2]{Clarke1983} and \cite[Theorem~8.12.2]{JDieudonne69}, one can show that $\partial_{G,\delta}^{2}f(x)$ is a compact convex set consisting of symmetric matrices.
\end{remark}

We are now ready to define the following notion of approximate second-order stationarity.
\begin{definition}
  Let $f\in \C1$ and let $\epsilon_1\ge 0$, $\epsilon_2\ge 0$ and $\delta > 0$. We say that $\bar x$ is an $(\epsilon_1,\epsilon_2,\delta)$-second-order stationary point of $f$ if
  \[
  \|\nabla f(\bar x)\|\le \epsilon_1\ \ {\rm and}\ \ \inf_{\|h\|=1}\suppf(h,\partial^2_{G,\delta}f(\bar x)h) \ge -\epsilon_2.
  \]
\end{definition}
In view of \cite[Theorem 3.1]{Hu84}, one can readily show that any local minimizer of an $f\in C^{1,1}(\R^d)$ is $(0,0,\delta)$-second-order stationary for any $\delta > 0$. 

Our next task is to design a zeroth-order algorithm to find $(\epsilon_1,\epsilon_2,\delta)$-second-order stationary points numerically, given $\epsilon_1>0$, $\epsilon_2>0$ and $\delta>0$.
Because of their strong theoretical guarantees, GS-based algorithms are among the most popular zeroth-order optimization algorithms. We will thus develop a new GS-based zeroth-order algorithm, and we aim at minimizing {\em Lipschitz differentiable} functions as in many previous works on GS-based algorithms (see, e.g., \cite{NesterovGS,LinZHengJordan22}) for succinctness. To this end, an important intermediate step is to investigate how the approximate second-order stationary points of the Gaussian smoothing $f_\sigma$ of $f$ are related to $(\epsilon_1,\epsilon_2,\delta)$-second-order stationary points. We start with some representation formulae for $\nabla^2 f_\sigma$. We would like to point out that the first equality in \eqref{GSHessian_Lip} was presented in \cite[Eq.~(4.3)]{Bala2022} under the additional assumption that $f$ has Lipschitz Hessian. We provide a detailed proof that only requires Lipschitz differentiability of $f$ for the convenience of the readers.

\begin{theorem} \label{twice_diff_lip}
    Suppose that $f: \mathbb{R}^d \rightarrow \mathbb{R}$ has Lipschitz continuous gradient with constant $L>0$, and $\sigma >0$. Then the Hessian of $f_{\sigma}$ is given by
\begin{align}
    \nabla^{2}f_{\sigma}(x)\!=\!\frac{1}{\sigma^{2}}\,\mathbb{E}_{u\sim\mathcal{N}(0,I)}\bigl[f(x+\sigma u)(uu^{T}-I)\bigr]
      \!=\!\frac{1}{\sigma}\,\mathbb{E}_{u\sim\mathcal{N}(0,I)}\bigl[u\nabla f(x+\sigma u)^{T}\bigr] \label{GSHessian_Lip}
\end{align}
and is well-defined and continuous.
\end{theorem}
\begin{proof}

Since $f$ has $L$-Lipschitz continuous gradient (and so does $-f$), we have from the Taylor's inequality that
\[
|f(x)|\le |f(0)| + \|\nabla f(0)\|\|x\| + \frac{L}{2}\|x\|^{2}\ \ \ \ \ \ \forall x\in \R^d.
\]
Then for any $r>0$, we have 
\[
\int_{\mathbb{B}_{r}}|f(x)|dx\le r^{d}\alpha(d)\cdot\left(|f(0)| + \|\nabla f(0)\|r + \frac{L}{2}r^{2}\right) =: {\frak U}(r),
\]
where $\alpha(d)$ is the volume of the $d$-dimensional unit ball $\B$. Since ${\frak U}(r)=O(r^{d+2})$ as $r\to\infty$, in view of \cite{Stein2011},\footnote{Specifically, see the last example on page 106.} we see that ${\frak F}(g):=\frac{1}{(2\pi)^\frac{d}{2}\sigma^{d}}\int_{\R^{d}}f(y)g(y)dy$ is a continuous linear functional on Schwartz space (i.e., a tempered distribution). Similarly, for each $i=1,\ldots,d$, ${\frak F}_i(g):=\frac{1}{(2\pi)^\frac{d}{2}\sigma^{d}}\int_{\R^{d}}y_if(y)g(y)dy$ is also a tempered distribution.

Based on the above facts, we now prove the first equality in \eqref{GSHessian_Lip} and the well-definedness of the expectation there. Our argument follows closely the proof of \cite[Theorem~2.3.20]{Grafakos2008}. From Lemma~\ref{lem053101}, we know that $\nabla f_{\sigma}$ is well-defined and continuous, and for any $x\in \R^d$, $t\in \R\backslash\{0\}$, and for any $i=j \in \{1,\ldots,d\}$,
\begin{align}
  &\frac{\nabla_{i}f_{\sigma}(x+te_{j})-\nabla_{i}f_{\sigma}(x)}{t}\notag\\
  &=\frac{1}{(2\pi)^\frac{d}{2}\sigma^{d+2}}\int_{\R^{d}}f(y)\frac{(y_{j}-(x_{j}+t))
  e^{-\frac{\|x+te_{j}-y\|^{2}}{2\sigma^{2}}}-(y_{j}-x_{j})e^{-\frac{\|x-y\|^{2}}{2\sigma^{2}}}}{t}dy\notag\\
  &=\frac1{\sigma^2} {\frak F}\left(\frac{(y_{j}-(x_{j}+t))e^{-\frac{\|x+te_{j}-y\|^{2}}{2\sigma^{2}}}-(y_{j}-x_{j})e^{-\frac{\|x-y\|^{2}}{2\sigma^{2}}}}{t}\right),\label{inintegral1}
\end{align}
while for any $i\neq j \in \{1,\ldots,d\}$,
\begin{align}
  &\frac{\nabla_{i}f_{\sigma}(x+te_{j})-\nabla_{i}f_{\sigma}(x)}{t}\notag\\
  &=\frac{1}{(2\pi)^\frac{d}{2}\sigma^{d+2}}\int_{\R^{d}}f(y)\frac{(y_{i}-x_{i})
  e^{-\frac{\|x+te_{j}-y\|^{2}}{2\sigma^{2}}}-(y_{i}-x_{i})e^{-\frac{\|x-y\|^{2}}{2\sigma^{2}}}}{t}dy\notag\\
  &=\frac1{\sigma^2} {\frak F}_i\left(\frac{e^{-\frac{\|x+te_{j}-y\|^{2}}{2\sigma^{2}}}- e^{-\frac{\|x-y\|^{2}}{2\sigma^{2}}}}{t}\right) - \frac{x_i}{\sigma^2}{\frak F}\left(\frac{e^{-\frac{\|x+te_{j}-y\|^{2}}{2\sigma^{2}}}-e^{-\frac{\|x-y\|^{2}}{2\sigma^{2}}}}{t}\right).\label{inintegral2}
\end{align}
Since
\begin{align*}
 &\frac{(y_{j}-(x_{j}+t))e^{-\frac{\|x+te_{j}-y\|^{2}}{2\sigma^{2}}}-(y_{j}-x_{j})e^{-\frac{\|x-y\|^{2}}{2\sigma^{2}}}}{t}\rightarrow 
 \left(\frac{(y_{j}-x_{j})^{2}}{\sigma^2}-1\right)e^{-\frac{\|x-y\|^{2}}{2\sigma^{2}}}\\
 &{\rm and}\ \ \ \ \ \frac{e^{-\frac{\|x+te_{j}-y\|^{2}}{2\sigma^{2}}}-e^{-\frac{\|x-y\|^{2}}{2\sigma^{2}}}}{t}\rightarrow 
 \frac{y_{j}-x_{j}}{\sigma^2}e^{-\frac{\|x-y\|^{2}}{2\sigma^{2}}}
\end{align*}
as $t \to 0$ in Schwartz space according to \cite[Exercise~2.3.5(a)]{Grafakos2008} and ${\frak F}$ is a tempered distribution, we conclude upon passing to the limit as $t \to 0$ in \eqref{inintegral1} and \eqref{inintegral2} that
\begin{align*}
  \nabla^{2} f_{\sigma}(x)&= \frac{1}{(2\pi)^\frac{d}{2}\sigma^{d+2}}\int_{\R^d}f(y)[(y-x)(y-x)^{T}/\sigma^{2}-I]e^{-\frac{\|x-y\|^{2}}{2\sigma^{2}}}dy\\
  &=\frac{1}{(2\pi)^\frac{d}{2}\sigma^{2}}\int_{\R^d}f(x+\sigma u)(uu^{T}-I)e^{-\frac{\|u\|^{2}}{2}}du.
\end{align*}
This proves the first equality in \eqref{GSHessian_Lip} and the well-definedness of the integral.

The continuity of $\nabla^{2} f_{\sigma}$ now follows immediately from the above integral representation and Lebesgue dominated convergence theorem.

Finally, it remains to prove the last equality in \eqref{GSHessian_Lip}. Since $f$ has $L$-Lipschitz continuous gradient, we have for any $t\neq 0$,
\begin{align}\label{ineq_bound_Lip}
 &|t|^{-1}|f(x+te_{j}+\sigma u)-f(x+\sigma u)|\le \sup_{0\le\theta\le1}\|\nabla f(x+(\theta t)e_{j}+\sigma u)\|\notag\\
 &\le\|\nabla f(0)\|+L\|x+\sigma u\|+L|t|.
\end{align}

Then we have
\begin{align*}
\nabla^{2}_{ij}f_{\sigma}(x)&=\lim_{t\rightarrow 0}\frac{\nabla_{i}f_{\sigma}(x+te_{j})-\nabla_{i}f_{\sigma}(x)}{t}\\
&=\lim_{t\rightarrow 0}\frac{1}{(2\pi)^{d/2}\sigma}\int_{\R^{d}}\frac{f(x+te_{j}+\sigma u)-f(x+\sigma u)}{t}u_{i}e^{-\frac{\|u\|^{2}}{2}}du\\
&=\frac{1}{(2\pi)^{d/2}\sigma}\int_{\R^{d}}\lim_{t\rightarrow 0}\frac{f(x+te_{j}+\sigma u)-f(x+\sigma u)}{t}u_{i}e^{-\frac{\|u\|^{2}}{2}}du\\
&=\frac{1}{(2\pi)^{d/2}\sigma}\int_{\R^{d}}\nabla_{j}f(x+\sigma u)u_{i}e^{-\frac{\|u\|^{2}}{2}}du,
\end{align*}
where the second equality follows from Lemma~\ref{lem053101}, and the third equality follows from \eqref{ineq_bound_Lip}, Lebesgue dominated convergence theorem and the fact that
\[
\int_{\mathbb{R}^{d}}(\|\nabla f(0)\|+L\|x+\sigma u\|+L)\|u\|e^{-\frac{\|u\|^{2}}{2}}du<+\infty.
\]
This completes the proof of \eqref{GSHessian_Lip}.
\end{proof}

We can now prove the following theorem relating $\nabla^2 f_\sigma(x)$ and $\partial^2_{G,\delta}f(x)$. 

\begin{theorem}[GS Hessian and Goldstein second-order $\delta$-subdifferential] \label{theorem_hessian}
    Assume that $f: \mathbb{R}^d \rightarrow \mathbb{R}$ has Lipschitz continuous gradient with constant $L >0$, and $\sigma > 0$. Then the following statements hold.
    \begin{enumerate}[{\rm (i)}]
  \item For any $x\in \R^d$ and $h\in\R^{d}$, one has
 \begin{align}\label{Hex}
  \langle h,\nabla^{2}f_{\sigma}(x)h\rangle=\langle h,\mathbb{E}_{u\sim \mathcal{N}(0,I)}[\nabla^{2}f(x+\sigma u)\cdot \mathbb{I}_{\mathfrak{D}_{\sigma}}(u)]h\rangle,
 \end{align}
 where $\mathfrak{D}_\sigma := \{u\in \R^d:\; \nabla f \mbox{ is differentiable at } x + \sigma u\}$ and $\mathbb{I}_{\mathfrak{D}_{\sigma}}$ is the characteristic function of $\mathfrak{D}_\sigma$.
  \item For any $\epsilon>0$ and $\delta>0$, it holds that
 \begin{align*}
  -2\epsilon+\lambda_{\min}[\nabla^{2}f_{\sigma}(x)]\le \inf_{\|h\|=1}\suppf(h,\partial_{G,\delta}^{2} f(x)h)
  \ \ \ \ \forall x\in \R^d\ \ {and}\ \ \forall \sigma \in (0,\bar{\sigma}],
\end{align*}
where
\begin{equation}\label{defsigma}
\begin{array}{c}
\bar{\sigma}:=\min\{1,\delta/H\},\ 
  H:=\sqrt{-d\cdot W_{-1}(-\eta_{1}^{\frac{2}{d}}/(2\pi e))},\\
  \eta_{1}:=\min\{\epsilon/(4L), (2\pi)^\frac{d}{2}-0.5\},
\end{array}
\end{equation}
and $W_{-1}$ is the negative real branch of the Lambert $W$ function,\footnote{Since $\eta_1\le (2\pi)^\frac{d}{2}-0.5$, we have $\eta_1^{2/d}/(2\pi e)< 1/e$ and hence $W_{-1}(-\eta_1^{2/d}/(2\pi e)) < -1$. Thus, $H\in (\sqrt{d},\infty)$.\label{footnote}}
 $\suppf(\cdot,\partial_{G,\delta}^{2} f(x)h)$ is the support function of the set $\partial_{G,\delta}^{2} f(x)h$.
\end{enumerate}
\end{theorem}
\begin{remark}[Simplified choice of $\bar\sigma$]\label{remark2025092201}
We present more explicit upper bounds for $\sigma$ in Theorem \ref{theorem_hessian}(ii). Let 
\begin{align*}
    0<\epsilon<\min\{5{L},1\}\ \ \text{and}\ \ 0<\delta<1.
\end{align*}
Now, using \eqref{defsigma} and following the discussions in \cite[Remark~3.7]{Pong2024} from \cite[Eq.~(3.28)]{Pong2024} to \cite[Eq.~(3.32)]{Pong2024} with $m = {\rm R}_1 = 0$ and ${\rm R}_2 = L$, we deduce that the inequality in Theorem \ref{theorem_hessian}(ii) holds whenever
\begin{equation}\label{2025092201}
 \sigma\le
 [\delta(4L)^{-1/d}/\sqrt{d\pi e}]\cdot\epsilon^{\frac{1}{d}}, 
\end{equation}
\end{remark}
\begin{proof}[Proof of Theorem~\ref{theorem_hessian}]
    (i) From Theorem~\ref{twice_diff_lip}, we know that $f_{\sigma}$ is twice continuously differentiable on $\mathbb{R}^{d}$.  Thus, for any fixed $x\in \R^d$ and $h\in \R^d$, one has
\begin{align}
  &\langle h,\nabla^{2}f_{\sigma}(x)h\rangle=\lim\limits_{\substack{ h'\rightarrow h \\ t\downarrow0}}\frac{f_{\sigma}(x+th')-f_{\sigma}(x)-t\langle\nabla f_{\sigma}(x),h'\rangle}{\frac{1}{2}t^{2}}\notag\\
  &=\lim\limits_{\substack{ h'\rightarrow h \\ t\downarrow0}}\frac{1}{\frac{1}{2}t^{2}}\mathbb{E}_{u\sim \mathcal{N}(0,I)}\left[f(x+th'+\sigma u) - f(x+\sigma u) -t\langle  \nabla f(x+\sigma u), h' \rangle \right]\notag\\
  &=\lim\limits_{\substack{ h'\rightarrow h \\ t\downarrow0}}\frac{1}{\frac{1}{2}t^{2}}\frac{1}{(2\pi)^{\frac{d}{2}}}\int_{{\frak D}_\sigma}\left[f(x+th'+\sigma u)-f(x+\sigma u)-t\langle\nabla f(x+\sigma u),h'\rangle\right]e^{-\frac{\|u\|^{2}}{2}}du,\label{h_dot_hassian}
\end{align}
where the second equality follows from the definition of $f_\sigma$ and \cite[Theorem~3.6(i)]{Pong2024}, and the last equality holds because of Rademacher's theorem.
Since $f$ has Lipschitz continuous gradient with constant $L$, we can obtain from \cite[Theorem 2.3]{Hu84} that for all $t > 0$, $h'\in \R^d$ and $u\in {\frak D}_\sigma$, 
there exists $\theta\in(0,1)$ such that
\begin{align*}
  f(x+th'+\sigma u)-f(x+\sigma u)-t\langle\nabla f(x+\sigma u),h'\rangle\in\frac{1}{2}t^{2}\langle\partial_{H}^{2}f(x+\sigma u+\theta th')h',h'\rangle.
\end{align*}
Then we see from \eqref{Equality} and \cite[Proposition~2.6.2(d)]{Clarke1983} that
\begin{align}
 &|f(x+th'+\sigma u)-f(x+\sigma u)-t\langle\nabla f(x+\sigma u),h'\rangle|\le\frac{L}{2}t^{2}\|h'\|^{2}.\label{in_subdiff}
\end{align}
Next, since $\nabla f$ is differentiable at $x + \sigma u$ whenever $u\in {\frak D}_\sigma$, standard argument shows that (see, e.g., the proof of \cite[Theorem~13.2(b)]{Tyrrell1998}) whenever $u\in {\frak D}_\sigma$,
\[
f(x+th'+\sigma u)-f(x+\sigma u)-t\langle\nabla f(x+\sigma u),h'\rangle = \frac12 t^2\langle h', \nabla^2 f(x+\sigma u)h'\rangle + o(t^2 \|h'\|^2),
\]
which means that
\[
\lim\limits_{\substack{ h'\rightarrow h \\ t\downarrow0}} \frac{f(x+th'+\sigma u)-f(x+\sigma u)-t\langle\nabla f(x+\sigma u),h'\rangle}{\frac12 t^2} = \langle h, \nabla^2 f(x+\sigma u)h\rangle.
\]
Using this together with \eqref{h_dot_hassian}, \eqref{in_subdiff} and Lebesgue dominated convergence theorem, we conclude that the desired result holds.

(ii) Fix any $\epsilon > 0$ and $\delta > 0$. Consider any $x\in \R^d$. Since $f$ has Lipschitz continuous gradient with constant $L > 0$, we see (from the last three displays) that
\begin{align}
  \|\nabla^{2}f(x+\sigma u)\|\le L\ \ \ \ \forall \sigma > 0 \ {\rm and}\ \forall u\in {\frak D}_\sigma.\label{eq033101}
\end{align}
We also observe for any $h$ with $\|h\|=1$ and any $M > 0$ that
\begin{align*}
 &\langle h,\mathbb{E}_{u\sim \mathcal{N}(0,I)}[\nabla^{2}f(x+\sigma u)\cdot \mathbb{I}_{\mathfrak{D}_{\sigma}}(u)]h\rangle\\
 &=\langle h,\mathbb{E}_{u\sim \mathcal{N}(0,I)}[\nabla^{2}f(x+\sigma u)\cdot \mathbb{I}_{\mathfrak{D}_{\sigma}\cap \mathbb{B}_{M}}(u)]h\rangle+\langle h,\mathbb{E}_{u\sim \mathcal{N}(0,I)}[\nabla^{2}f(x+\sigma u)\cdot \mathbb{I}_{\mathfrak{D}_{\sigma}\cap \mathbb{B}_{M}^{c}}(u)]h\rangle\\
 &=\frac{\langle h,\mathbb{E}_{u\sim \mathcal{N}(0,I)}[\nabla^{2}f(x+\sigma u)\cdot \mathbb{I}_{\mathfrak{D}_{\sigma}\cap \mathbb{B}_{M}}(u)]h\rangle}{\mathbb{E}_{u\sim \mathcal{N}(0,I)}[\mathbb{I}_{\mathfrak{D}_{\sigma}\cap\mathbb{B}_{M}}(u)]}+\langle h,\mathbb{E}_{u\sim \mathcal{N}(0,I)}[\nabla^{2}f(x+\sigma u)\cdot \mathbb{I}_{\mathfrak{D}_{\sigma}\cap \mathbb{B}_{M}^{c}}(u)]h\rangle\\
 &~~+\left(1-\frac{1}{\mathbb{E}_{u\sim \mathcal{N}(0,I)}[\mathbb{I}_{\mathfrak{D}_{\sigma}\cap\mathbb{B}_{M}}(u)]}\right)\langle h,\mathbb{E}_{u\sim \mathcal{N}(0,I)}[\nabla^{2}f(x+\sigma u)\cdot \mathbb{I}_{\mathfrak{D}_{\sigma}\cap \mathbb{B}_{M}}(u)]h\rangle.
\end{align*}
Now, using \eqref{eq033101} and following the proof of \cite[Theorem~3.6(ii)]{Pong2024} from \cite[Eq.~(3.19)]{Pong2024} to \cite[Eq.~(3.25)]{Pong2024} with $m={\rm R}_{1}=0$ and ${\rm R}_{2}=L$, we can similarly derive bounds for the second and third summands on the last two lines of the above display. Specifically, for the fixed $\epsilon$ and $\delta>0$, it holds for $M=H$ and any $\sigma \in (0,\bar\sigma]$ (see \eqref{defsigma} for the definitions of $H$ and $\bar\sigma$) that
\begin{align}
  &\langle h,\mathbb{E}_{u\sim \mathcal{N}(0,I)}[\nabla^{2}f(x+\sigma u)\cdot \mathbb{I}_{\mathfrak{D}_{\sigma}}(u)]h\rangle\notag\\
  &\le\frac{\langle h,\mathbb{E}_{u\sim \mathcal{N}(0,I)}[\nabla^{2}f(x+\sigma u)\cdot \mathbb{I}_{\mathfrak{D}_{\sigma}\cap \mathbb{B}_{M}}(u)]h\rangle}{\mathbb{E}_{u\sim \mathcal{N}(0,I)}[\mathbb{I}_{\mathfrak{D}_{\sigma}\cap\mathbb{B}_{M}}(u)]}+2\epsilon\ \ \ \ \forall h \in \{w:\|w\|=1\}.\label{eq042201}
\end{align}
Next, observe that $\sigma \in (0,\bar\sigma]$ and $M = H$ implies that $\sigma \le \delta/M$. Thus, for any $\sigma \in (0,\bar\sigma]$ and $M = H$,
we have from the definition of the Goldstein second-order $\delta$-subdifferential that
\begin{align}
 \frac{\langle h,\mathbb{E}_{u\sim \mathcal{N}(0,I)}[\nabla^{2}f(x+\sigma u)\cdot \mathbb{I}_{\mathfrak{D}_{\sigma}\cap \mathbb{B}_{M}}(u)]h\rangle}{\mathbb{E}_{u\sim \mathcal{N}(0,I)}[\mathbb{I}_{\mathfrak{D}_{\sigma}\cap\mathbb{B}_{M}}(u)]}\in\langle\partial_{G,\delta}^{2} f(x)h,h\rangle\ \ \ \ \forall h \in \{w:\|w\|=1\}.\label{eq042202}
\end{align}
Combining \eqref{eq042202}, \eqref{eq042201} and \eqref{Hex}, we can obtain the desired result.
\end{proof}

\section{GS-based method for $(\epsilon_1,\epsilon_2,\delta)$-2nd-order stationary points of Lipschitz differentiable functions}

In this section, we consider the optimization problem
\begin{equation}\label{problem}
	\begin{array}{rl}
		\displaystyle{\min_{x\in \mathbb{R}^d}} & f(x),
	\end{array}		
\end{equation}	
where $f$ has $L$-Lipschitz continuous gradient for some $L > 0$ but is not necessarily twice continuously differentiable, and $f_* := \inf f > -\infty$. We focus on the case where the function value of $f$ is available only via a blackbox or costly computational routine, and the gradient is too costly to obtain, and we will develop a zeroth-order algorithm for obtaining an $(\epsilon_1,\epsilon_2,\delta)$-second-order stationary point of $f$ in \eqref{problem}.

Motivated by Theorem~\ref{theorem_hessian} which builds a connection between the Hessian of Gaussian smoothing and the Goldstein second-order $\delta$-subdifferential, we will develop our algorithm based on Gaussian smoothing. 
Our algorithm is presented as Algorithm~\ref{algorithm 01} below. It can be viewed as a variant of \cite[Algorithm~7]{Bala2022}, which was proposed for \eqref{problem} with a twice continuously differentiable $f$ whose Hessian is Lipschitz continuous. Similarly to \cite[Algorithm~7]{Bala2022}, we adopt sample average approximations to approximate $\nabla f_\sigma$ and $\nabla^2 f_\sigma$ every iteration (leveraging \eqref{GSgradient} and \eqref{GSHessian_Lip}), and find the next iterate by solving a suitable cubic regularization subproblem.\footnote{Recall that there are many efficient methods for solving cubic regularization subproblems; see, e.g., \cite{CartisGouldToint11,Lieder20}.} However, unlike \cite[Algorithm~7]{Bala2022} (see \cite[Theorem~9]{Bala2022}), we do not fix the sample sizes for approximating $\nabla f_\sigma$ and $\nabla^2 f_\sigma$ according to a prescribed tolerance. Indeed, our algorithm does not require the users to specify a tolerance before starting the algorithm; in this regard, we follow \cite{LeiShan25} and increase our sample sizes for approximating $\nabla f_\sigma$ and $\nabla^2 f_\sigma$ every iteration. Furthermore, in contrast to the constant smoothing parameter $\sigma$ in \cite[Algorithm~7]{Bala2022}, we adopt a vanishing $\{\sigma_k\}$, making our approach an instance of homotopy methods (see, e.g., \cite{IwaWangItoTakeda22,Starnes2023}). Finally, we also need to point out that our definition of $h_k$ in \eqref{defsq} differs from the classical one used in the literature of cubic regularization methods. Specifically, we introduce an additional scaling of $(1+\|x^k\|)^3$; as we will see later in our convergence analysis, this extra scaling is introduced to account for a similar scaling that arises when we derive an upper bound for $\mathbb{E}_{k}\left[\|\nabla f_{\sigma_{k}}(x^{k})-v^{k}\|^{2}\right]$ in Lemma~\ref{lem2025080501}; note that this latter scaling in the upper bound is inevitable in general, see Remark~\ref{rem:dependence} below.

\begin{algorithm}[h]
		\caption{\!: Zeroth-order method for solving problem \eqref{problem}}
		\begin{algorithmic}
			\STATE
			{\bf Input:} Initial point $x^0\in \mathbb{R}^{d}$, real numbers $\alpha>0$, $\beta>0$, $\rho > 0$,   $\gamma\in (0,1)$. Set $\sigma_0 = \gamma$ and $\sigma_k = \gamma k^{-\rho}$ for all $k \ge 1$. Set iteration counter $k = 0$.
			
			{\bf Step 1:} Let
            \[
            N_{k}=\lceil(k+1)^{\frac{4}{3}\beta}\rceil,\
            J_{k}=\lceil(k+1)^{\frac{2}{3}\alpha}\rceil,
            \]
        and
generate $\max\{N_{k},J_k\}$ i.i.d. $u_{i}^{k}\sim \mathcal{N}(0,I)$ for $i=1,\ldots,\max\{N_k,J_k\}$, and set
\begin{align}
v^{k}&=\frac{1}{N_{k}\sigma_{k}}\sum_{i=1}^{N_{k}}(f(x^{k}+\sigma_{k} u_{i}^{k})-f(x^{k}))u_{i}^{k},\label{vkdef}\\
H^{k}&=\frac{1}{2J_{k}\sigma_{k}^{2}}\sum_{i=1}^{J_{k}}(f(x^{k}+\sigma_{k} {u}_{i}^{k})+f(x^{k}-\sigma_{k} {u}_{i}^{k})-2f(x^{k}))({u}_{i}^{k}({u}_{i}^{k})^{T}-I).\label{Hkdef}
\end{align}

{\bf Step 2:} Compute
			\begin{equation}\label{defsq}
				x^{k+1} \in \Argmin_{x\in\mathbb{R}^{d}} h_{k}(x),
			\end{equation}
where $h_{k}(x):=\langle v^{k},x-x^{k}\rangle+\frac{1}{2}\langle H^{k}(x-x^{k}),x-x^{k}\rangle+\frac{L\sqrt{d}}{6\sigma_{k}}(1+\|x^{k}\|)^{3}\|x-x^k\|^3$. Update $k\leftarrow k+1$ and go to {\bf Step 1}.
\end{algorithmic}
\label{algorithm 01}
\end{algorithm}

For the rest of this section, we will study the iteration complexity for obtaining an $(\epsilon_1,\epsilon_2,\delta)$-second-order stationary point using our algorithm. We first prove some auxiliary lemmas in Section~\ref{sec:lemma}. The complexity results are then derived in Section~\ref{sec:alg}.

\subsection{Auxiliary lemmas}\label{sec:lemma}

In this subsection, we present several auxiliary lemmas for our subsequent convergence analysis. These include some structural properties of $f_\sigma$ and some basic inequalities satisfied by the sequence $\{x^k\}$ generated by Algorithm~\ref{algorithm 01}. We start with the following descent lemma.
\begin{lemma}[Cubic descent lemma]\label{lem:cubic_descent}
    Suppose that $f$ has Lipschitz gradient with constant $L > 0$ and $\sigma >0$. Then it holds that
    \[
        f_{\sigma}(y) \le f_{\sigma}(x) +  \langle \nabla f_{\sigma}(x), y-x \rangle + \frac{1}{2}\langle \nabla^2f_{\sigma}(x)(y-x), y-x \rangle + \frac{L \sqrt{d}}{6\sigma}\|y-x\|^3\ \ \ \ \forall x,y\in \R^d.
    \]
\end{lemma}
\begin{proof}
    Notice that for any $x, y \in \R^d$, we have
    \begin{align}
        \| \nabla^2 f_{\sigma} (y) - \nabla^2 f_{\sigma}(x)\|& = \frac{1}{\sigma}  \left\|\E_{u\sim\mathcal{N}(0,I)} [(\nabla f(y+\sigma u) - \nabla f(x+ \sigma u))u^T] \right\|\notag \\
        &\le \frac{L\|y-x\|}{\sigma}\E_{u\sim\mathcal{N}(0,I)} \|u\|, \label{diff_hessian}
    \end{align}
    where the last inequality holds because $f$ has Lipschitz gradient with constant $L$. Therefore, we have 
    \begin{align}
       \left\|\nabla f_{\sigma} (y) -\nabla f_{\sigma}(x) - \nabla^2 f_{\sigma}(x) (y-x)\right\| &\le \int_0^1 \| \nabla^2 f_{\sigma}(x+ \tau(y-x)) - \nabla^2 f_{\sigma}(x)\|d\tau \cdot \|y -x\| \nonumber \\
       &\le \frac{L\|y -x\|^2}{2\sigma} \E_{u\sim\mathcal{N}(0,I)}\|u\| \le \frac{L\|y -x\|^2\sqrt{d}}{2\sigma}, \nonumber
    \end{align}
    where the second inequality holds because of \eqref{diff_hessian} and the last inequality follows from \cite[Lemma~1]{NesterovGS}. Then, we have
    \begin{align}
         &\bigg| f_{\sigma}(y)-f_{\sigma}(x)-\langle\nabla f_{\sigma}(x),y-x\rangle-\frac{1}{2}\langle\nabla^{2}f_{\sigma}(x)(y-x),y-x\rangle \bigg| \nonumber \\
         &=\bigg| \int_0^1 \left(\langle \nabla f_\sigma(x + \tau(y-x)) - \nabla f_\sigma(x),y-x\rangle - \tau\langle\nabla^2 f_\sigma(x)(y-x),y-x\rangle\right) d\tau \bigg| \nonumber \\
         &\le\int_{0}^{1}\|\nabla f_{\sigma}(x+\tau(y-x))-\nabla f_{\sigma}(x)-\tau\nabla^{2}f_{\sigma}(x)(y-x)\|d\tau\cdot\|y-x\|  \nonumber  \\
         & \le \frac{L\|y-x\|^3\sqrt{d}}{2\sigma} \int_0^1 \tau^2 d\tau = 
         \frac{L\|y-x\|^3\sqrt{d}}{6\sigma}.  \nonumber    
    \end{align}
\end{proof}

The next lemma concerns the scaling that appears in the summand in \eqref{Hkdef}.
\begin{lemma}\label{lem2025102601}
Suppose that $f$ has Lipschitz gradient with constant $L > 0$ and $\sigma >0$. Then, for any $x\in \R^d$ and any $p\ge 1$, we have
 \begin{align*}
\mathbb{E}_{u\sim {\cal N}(0,I)}\left[\left|\frac{f(x+\sigma u)+f(x-\sigma u)-2f(x)}{2\sigma^{2}}\right|^{p}\right]\le \frac{L^p(d+2p)^p}{2^p}.
 \end{align*}
\end{lemma}
\begin{proof}
Notice that
\begin{align*}
&|f(x+\sigma u)+f(x-\sigma u)-2f(x)|=|f(x+\sigma u)-f(x)+f(x-\sigma u)-f(x)|\\
&=\left|\int_{0}^{\sigma}\langle \nabla f(x+tu),u\rangle dt+\int_{0}^{\sigma}\langle \nabla f(x-tu),-u\rangle dt\right|\\
&=\left|\int_{0}^{\sigma}\langle \nabla f(x+tu)-\nabla f(x-tu),u\rangle dt\right|\le\int_{0}^{\sigma} \|\nabla f(x+tu)-\nabla f(x-tu)\| dt\cdot \|u\|\\
&\le 2L\|u\|^2 \int_0^{\sigma} t dt = L\sigma^2\|u\|^2.
\end{align*}
This implies that
\begin{align*}
&\mathbb{E}_{u\sim {\cal N}(0,I)}\left[\left|\frac{f(x+\sigma u)+f(x-\sigma u)-2f(x)}{2\sigma^{2}}\right|^{p}\right]\\
&\le \E_{u\sim {\cal N}(0,I)} \left[ \left( \frac{L\sigma^2\|u\|^2}{2\sigma^2} \right)^p\right] = \frac{L^p}{2^p}\E_{u\sim {\cal N}(0,I)} [\|u\|^{2p}] \le \frac{L^p(d+2p)^p}{2^p},
\end{align*}
where the last inequality follows from  \cite[Lemma~1]{NesterovGS}.
\end{proof}

The next lemma is a well-known result concerning the optimality conditions for the subproblem \eqref{defsq}. We include it here for ease of reference.
\begin{lemma}[Optimality conditions for \eqref{defsq}]
  Consider \eqref{problem} and let $\{x^{k}\}$ be generated by Algorithm \ref{algorithm 01}. Then one has for all $k\ge 0$ that
  \begin{align}
    v^{k}+H^{k}(x^{k+1}-x^{k})+\frac{L\sqrt{d}}{2\sigma_{k}}(1+\|x^{k}\|)^{3}\|x^{k+1}-x^{k}\|(x^{k+1}-x^{k})=0\label{eq022301}
  \end{align}
  and
  \begin{align}
    H^{k}+\frac{L\sqrt{d}}{2\sigma_{k}}(1+\|x^{k}\|)^{3}\|x^{k+1}-x^{k}\|I\succeq0,\label{eq022302}
  \end{align}
  where $v^k$ and $H^k$ are defined in \eqref{vkdef} and \eqref{Hkdef}, respectively.
\end{lemma}
\begin{proof}
Relation \eqref{eq022301} follows from the first-order optimality condition of the subproblem \eqref{defsq}, while \eqref{eq022302} was proved in \cite[Proposition~1]{NesPol06}.
\end{proof}

The next lemma gives some bounds on the quality of the sample average approximations in \eqref{vkdef} and \eqref{Hkdef} for $\nabla f_{\sigma_k}(x^k)$ and $\nabla^2 f_{\sigma_k}(x^k)$, respectively. For its proof, we need to use the following simple bounds concerning a $u\sim {\cal N}(0,I)$:
\begin{align}\label{eq2025082504}
&\mathbb{E}_{u\sim {\cal N}(0,I)}\left[\|uu^{T}-I\|_{F}^{4}\right]=\mathbb{E}_{u\sim {\cal N}(0,I)}\left[(\|u\|^{4}-2\|u\|^{2}+d)^{2}\right]\notag\\
&=\mathbb{E}_{u\sim {\cal N}(0,I)}\left[((\|u\|^2-1)^2+d-1)^{2}\right]\overset{\rm (a)}\le\mathbb{E}_{u\sim {\cal N}(0,I)}\left[(\|u\|^{4}+d)^{2}\right]\notag\\
&=\mathbb{E}_{u\sim {\cal N}(0,I)}\left[\|u\|^{8}\right]+2d\mathbb{E}_{u\sim {\cal N}(0,I)}\left[\|u\|^{4}\right]+d^{2}\overset{\rm (b)}\le(8+d)^{4}+2d(4+d)^{2}+d^{2}\notag\\
&\le[(8+d)^{2}+d]^{2}\le(d+9)^{4},
\end{align}
and
\begin{align}\label{eq202508250422}
&\mathbb{E}_{u\sim {\cal N}(0,I)}\left[\|uu^{T}-I\|_{F}^{8}\right]=\mathbb{E}_{u\sim {\cal N}(0,I)}\left[(\|u\|^{4}-2\|u\|^{2}+d)^{4}\right]\notag\\
&=\mathbb{E}_{u\sim {\cal N}(0,I)}\left[((\|u\|^2-1)^2+d-1)^{4}\right]\overset{\rm (a)}\le\mathbb{E}_{u\sim {\cal N}(0,I)}\left[(\|u\|^{4}+d)^{4}\right]\notag\\
&\overset{\rm (c)}\le4\mathbb{E}_{u\sim {\cal N}(0,I)}\left[(\|u\|^{8}+d^2)^{2}\right]=4\mathbb{E}_{u\sim {\cal N}(0,I)}\left[\|u\|^{16}\right]+8d^2\mathbb{E}_{u\sim {\cal N}(0,I)}\left[\|u\|^{8}\right]+4d^{4}\notag\\
&\overset{\rm (b)}\le4[(16+d)^{8}+2d^2(8+d)^{4}+d^{4}]\le16(d+16)^{8},
\end{align}
where (a) holds because $d\ge 1$ and $(\|u\|^2-1)^2\le \|u\|^4 + 1$, (b) follows from \cite[Lemma~1]{NesterovGS} and (c) follows from the relation $(a+b)^2\le 2(a^2 + b^2)$ for any $a$, $b\in \mathbb{R}$.

From now on, we $\mathbb{E}_k$ to denote the conditional expectation given the entire history of the algorithm prior to the generation of the random samples at the $k$-th iteration. In particular, $x^k$, $\sigma_k$, $N_k$, and $J_k$ are fixed under $\mathbb{E}_k$, while the random samples generated at iteration $k$ are independent of the past.

\begin{lemma}[Bounds for sample average approximations]\label{lem2025080501}
Consider problem \eqref{problem}. Let $k\ge 0$ and $x^k$ be generated at the beginning of the $k$th iteration of Algorithm \ref{algorithm 01}. Let $v^k$ and $H^k$ be generated by \eqref{vkdef} and \eqref{Hkdef}, respectively. Then it holds that
\begin{align}
\mathbb{E}_{k}\left[\|\nabla f_{\sigma_{k}}(x^{k})-v^{k}\|^{2}\right] 
&\le \frac{D_1}{(k+1)^{\frac{4}{3}\beta}}\big(1+\|x^k\| \big)^2, \label{var_vk}
\end{align}
\begin{align}
\mathbb{E}_{k}\left[\|\nabla^{2} f_{\sigma_{k}}(x^{k})-H^{k}\|_{F}^{2}\right]\leq \frac{D_2}{(k+1)^{\frac{2}{3}\alpha}},\label{2ndexpectation}
\end{align}
and
\begin{align}
\mathbb{E}_{k}\left[\|\nabla^{2} f_{\sigma_{k}}(x^{k})-H^{k}\|_{F}^{3}\right]\leq \frac{D_3}{(k+1)^{\alpha}},\label{thirdexpectation}
\end{align}
where $D_1 := \frac{L^2(d+6)^3}{2}\gamma^2 + 4(d+4)D_0 $, $D_0 := \max \left\{\|\nabla f(0)\|^2, L^2 \right\}$, $D_2 := L^{2}(d+9)^{4}$ and $D_3 := 6dL^{3}(d+9)^{2}(d+16)^{4}$.
\end{lemma}
\begin{remark}[Dependence on $\|x^k\|^2$]\label{rem:dependence}
Note that the right hand side of \eqref{var_vk} depends on $\|x^k\|^2$. We claim that such a dependence cannot be relaxed in general. To see this, consider $f:\R\to \R$ given by $f(x)=x^{2}$. Then $f$ has $2$-Lipschitz continuous gradient and a direct computation shows
\begin{align*}
&N_k\mathbb{E}_{k}\left[(f'_{\sigma_{k}}(x^{k})-{v}^{k})^{2}\right] = N_k\mathbb{E}_{k}\left[\left(f'_{\sigma_{k}}(x^{k})-\frac1{N_k}\sum_{i=1}^{N_k}\widehat v_i^k\right)^{2}\right]\\
 &=\mathbb{E}_{u\sim {\cal N}(0,1)}\left[(f'_{\sigma_{k}}(x^{k})-\widehat{v}^{k})^{2}\right]=\mathbb{E}_{u\sim {\cal N}(0,1)}[(\widehat{v}^{k})^{2}]-(f'_{\sigma_{k}}(x^{k}))^2\\
&=\mathbb{E}_{u\sim {\cal N}(0,1)}\left[\left(\frac{f(x^{k}+\sigma_{k}u)-f(x^{k})}{\sigma_{k}}u\right)^{2}\right]-(\mathbb{E}_{u\sim {\cal N}(0,1)}[f'(x^{k}+\sigma_{k} u)])^{2}\\
&=\mathbb{E}_{u\sim {\cal N}(0,1)}\left[\left|\frac{(x^{k}+\sigma_{k}u)^{2}-(x^{k})^{2}}{\sigma_{k}}u\right|^{2}\right]-\left|\mathbb{E}_{u\sim {\cal N}(0,1)}[2(x^{k}+\sigma_{k} u)]\right|^{2}\\
&=\mathbb{E}_{u\sim {\cal N}(0,1)}\left[\left|(2x^{k}+\sigma_{k} u)u^{2}\right|^{2}\right]-4(x^{k})^{2}\\
&=\mathbb{E}_{u\sim {\cal N}(0,1)}\left[(4(x^{k})^{2}+4\sigma_{k} ux^{k}+\sigma_{k}^{2}u^{2})u^{4}\right]-4(x^{k})^{2}\\
&=4\left[\mathbb{E}_{u\sim {\cal N}(0,1)}(u^{4})-1\right](x^{k})^{2}+\sigma_{k}^{2}\mathbb{E}_{u\sim {\cal N}(0,1)}[u^{6}]=8(x^{k})^{2}+15\sigma_{k}^{2},
\end{align*}
where $\widehat{v}_i^{k}:=\sigma_{k}^{-1}[f(x^{k}+\sigma_{k} u^k_i)-f(x^{k})]u^k_i$,   $\widehat{v}^{k}:=\sigma_{k}^{-1}[f(x^{k}+\sigma_{k} u)-f(x^{k})]u$ with $u\sim\mathcal{N}(0,1)$, the second equality holds because $u_i^k$ are i.i.d. samples with $\mathbb{E}_{k}[\widehat v_i^k] = f'_{\sigma_k}(x^k)$ (see \eqref{GSgradient}), the third equality holds because $\mathbb{E}_{k}[\widehat v^k] = f'_{\sigma_k}(x^k)$  (see \eqref{GSgradient}), the interchange of differentiation and expectation is valid in the fourth equality thanks to \cite[Theorem~3.6(i)]{Pong2024}, and we used \cite[Example 21.1]{Bill2008} in the last equality. This shows that a bound depending on $\|x^k\|^2$ is inevitable in general.
\end{remark}
\begin{proof}[Proof of Lemma~\ref{lem2025080501}]
Notice that
\begin{align}\label{eq2025082501}
&\mathbb{E}_{k}\left[\|\nabla f_{\sigma_{k}}(x^{k})-v^{k}\|^{2}\right]
=\mathbb{E}_{k}\left[\left\|\nabla f_{\sigma_{k}}(x^{k})-\frac{\sum_{i=1}^{N_{k}}\widehat{v}_{i}^{k}}{N_{k}}\right\|^{2}\right]\notag\\
&=\frac{1}{(N_{k})^{2}}\mathbb{E}_{k}\left[\left\|\sum_{i=1}^{N_{k}}(\nabla f_{\sigma_{k}}(x^{k})-\widehat{v}_{i}^{k})\right\|^{2}\right]
=\frac{1}{(N_{k})^{2}}\sum_{i=1}^{N_{k}}\mathbb{E}_{k}[\|(\nabla f_{\sigma_{k}}(x^{k})-\widehat{v}_{i}^{k}\|^2]\notag\\
&=\frac{1}{N_{k}}\mathbb{E}_{u\sim {\cal N}(0,I)}[\|\nabla f_{\sigma_{k}}(x^{k})-\widehat{v}^{k}\|^{2}] =\frac{1}{N_{k}}\left[\mathbb{E}_{u\sim {\cal N}(0,I)}(\|\widehat{v}^{k}\|^{2})-\|\nabla f_{\sigma_{k}}(x^{k})\|^{2}\right]\notag\\
&\leq \frac{1}{N_{k}}\mathbb{E}_{u\sim {\cal N}(0,I)}[\|\widehat{v}^{k}\|^{2}]\le \frac{\sigma_k^2}{2N_k}L^2(d+6)^3+ \frac{2(d+4)}{N_k}\|\nabla f(x^k) \|^2,
\end{align}
where $\widehat{v}_{i}^{k}:=\sigma_{k}^{-1}[f(x^{k}+\sigma_{k} u_{i}^{k})-f(x^{k})]u_{i}^{k}$, $\widehat{v}^{k}:=\sigma_{k}^{-1}[f(x^{k}+\sigma_{k} u)-f(x^{k})]u$ with $u\sim\mathcal{N}(0,I)$, the third and fourth equalities hold because $u_i^k$ are i.i.d. samples with $\mathbb{E}_{k}[\widehat v_i^k] = \nabla f_{\sigma_k}(x^k)$ (see \eqref{GSgradient}), the fifth equality holds because $\mathbb{E}_{u\sim {\cal N}(0,I)}[\widehat v^k] = \nabla f_{\sigma_k}(x^k)$  (see \eqref{GSgradient}) and the last inequality holds because of \cite[Theorem~4]{NesterovGS}.\footnote{Although this theorem was stated for a convex $f$, the inequality we need (i.e., the first inequality in \cite[Eq.~(35)]{NesterovGS}) does not rely on convexity.} Since $f$ has Lipschitz gradient with constant $L$, we have
\[
    \|\nabla f(x^k) \| \le \|\nabla f(0)\| + L\|x^k\|.
\]
Using this together with \eqref{eq2025082501}, we have 
\[
    \mathbb{E}_{k}\left[\|\nabla f_{\sigma_{k}}(x^{k})-v^{k}\|^{2}\right]\leq \frac{\sigma_k^2}{2(k+1)^{\frac{4}{3}\beta}}L^2(d+6)^3+ \frac{4(d+4)}{(k+1)^{\frac{4}{3}\beta}}\big(\| \nabla f(0)\|^2+ L^2\|x^k\|^2 \big).
\]
Now, in view of the definition of $D_0$, we have 
\begin{align*}
    \mathbb{E}_{k}\left[\|\nabla f_{\sigma_{k}}(x^{k})-v^{k}\|^{2}\right] &\leq \frac{\sigma_k^2}{2(k+1)^{\frac{4}{3}\beta}}L^2(d+6)^3+ \frac{4(d+4)}{(k+1)^{\frac{4}{3}\beta}}D_0\big(1+\|x^k\|^2 \big)  \\
    &\le \frac{\gamma^2}{2(k+1)^{\frac{4}{3}\beta}}L^2(d+6)^3+ \frac{4(d+4)}{(k+1)^{\frac{4}{3}\beta}}D_0\big(1\!+\!\|x^k\|^2 \big)\\
    &\le \frac{D_1}{(k+1)^{\frac{4}{3}\beta}}\big(1\!+\!\|x^k\|^2 \big) \le \frac{D_1}{(k+1)^{\frac{4}{3}\beta}}\big(1\!+\!\|x^k\|\big)^2, 
\end{align*}
where the second inequality holds because $\sigma_k \le \gamma$ for all $k\ge 0$, and the third inequality holds because of the definition of $D_1$. This proves \eqref{var_vk}.

Next, we prove the bound for $\mathbb{E}_{k}\left[\|\nabla^{2} f_{\sigma_{k}}(x^{k})-H^{k}\|_{F}^{2}\right]$. Similarly to \eqref{eq2025082501}, we can
derive
\begin{align}
\mathbb{E}_{k}\left[\|\nabla^{2} f_{\sigma_{k}}(x^{k})-H^{k}\|_{F}^{2}\right]\leq \frac{1}{J_{k}}\mathbb{E}_{u\sim\mathcal{N}(0,I)}[\|\widehat{H}^{k}\|_{F}^{2}],\label{1stboundH}
\end{align}
where $\widehat{H}^{k}:=\frac{1}{2\sigma_{k}^{2}}[f(x^{k}+\sigma_{k} u)+f(x^{k}-\sigma_{k} u)-2f(x^{k})](uu^{T}-I)$ with $u\sim\mathcal{N}(0,I)$, thanks to the fact that $\mathbb{E}_{u\sim\mathcal{N}(0,I)}[\widehat H^k] = \nabla^{2} f_{\sigma_{k}}(x^{k})$ (see the first equality in \eqref{GSHessian_Lip}). Notice that
\begin{align}\label{eq2025082502}
&\mathbb{E}_{u\sim\mathcal{N}(0,I)}[\|\widehat{H}^{k}\|_{F}^{2}]=\mathbb{E}_{u\sim\mathcal{N}(0,I)}\left[\left|\frac{f(x^{k}+\sigma_{k} u)+f(x^{k}-\sigma_{k} u)-2f(x^{k})}{2\sigma_{k}^{2}}\right|^{2}\|uu^{T}-I\|_{F}^{2}\right]\notag\\
&\le\left[\mathbb{E}_{u\sim\mathcal{N}(0,I)}\left(\left|\frac{f(x^{k}+\sigma_{k} u)+f(x^{k}-\sigma_{k} u)-2f(x^{k})}{2\sigma_{k}^{2}}\right|^{4}\right)\right]^{\frac{1}{2}}\notag\\
& \ \ \ \ \ \times
\left[\mathbb{E}_{u\sim\mathcal{N}(0,I)}\left(\|uu^{T}-I\|_{F}^{4}\right)\right]^{\frac{1}{2}}.
\end{align}
On the other hand, by Lemma \ref{lem2025102601}, we know that
\begin{align}\label{eq2025082503}
\mathbb{E}_{u\sim\mathcal{N}(0,I)}\left[\left|\frac{f(x^{k}+\sigma_{k} u)+f(x^{k}-\sigma_{k} u)-2f(x^{k})}{2\sigma_{k}^{2}}\right|^{4}\right]\le 2^{-4}L^{4}(d+8)^{4}.
\end{align}
Thus, combining \eqref{eq2025082502}, \eqref{eq2025082503} and \eqref{eq2025082504} with \eqref{1stboundH}, one has
\begin{align}
\mathbb{E}_{k}\left[\|\nabla^{2} f_{\sigma_{k}}(x^{k})-H^{k}\|_{F}^{2}\right]\leq \frac{1}{J_k}\cdot 2^{-2}L^{2}(d+8)^{2}(d+9)^{2}\le \frac{1}{J_k}\cdot L^{2}(d+9)^{4}.\label{2nddisplaying}
\end{align}
The relation \eqref{2ndexpectation} now follows immediately from the above display and the definition of $J_k$.

Finally, we prove the bound for $\mathbb{E}_{k}\left[\|\nabla^2 f_{\sigma_{k}}(x^{k})-H^{k}\|_F^{3}\right]$. Notice that
\begin{align}\label{eq2025081101}
\mathbb{E}_{k}\left[\|\nabla^2 f_{\sigma_{k}}(x^{k})-H^{k}\|_{F}^{3}\right]&\leq \left[\mathbb{E}_{k}(\|\nabla^{2}f_{\sigma_{k}}(x^{k})-H^{k}\|_{F}^{2})\right]^{1/2}
\cdot \left[\mathbb{E}_{k}(\|\nabla^{2}f_{\sigma_{k}}(x^{k})-H^{k}\|_{F}^{4})\right]^{1/2}\notag\\
&\le \frac{L(d+9)^{2}}{\sqrt{J_{k}}}\cdot \left[\mathbb{E}_{k}(\|\nabla^{2}f_{\sigma_{k}}(x^{k})-H^{k}\|_{F}^{4})\right]^{1/2},
\end{align}
where the last inequality follows from \eqref{2nddisplaying}. On the other hand, we know that
\begin{align}\label{eq2025081102}
&\mathbb{E}_{k}[\|\nabla^{2}f_{\sigma_{k}}(x^{k})-H^{k}\|_{F}^{4}]\notag\\
&=\frac{1}{J_{k}^{4}}\mathbb{E}_{k}\left[\left\|\sum_{i=1}^{J_{k}}(\nabla^{2}f_{\sigma_{k}}(x^{k})
-\widehat{H}_{i}^{k})\right\|_{F}^{4}\right]=\frac{1}{J_{k}^{4}}\mathbb{E}_{k}\left[\sum_{p,q=1}^{d}\left(\sum_{i=1}^{J_{k}}(\nabla^{2}f_{\sigma_{k}}(x^{k})-\widehat{H}_{i}^{k})_{p,q}\right)^{2}\right]^{2}\notag\\
&\le\frac{d^{2}}{J_{k}^{4}}\mathbb{E}_{k}\!\left[\sum_{p,q=1}^{d}\!\!\left(\sum_{i=1}^{J_{k}}(\nabla^{2}f_{\sigma_{k}}(x^{k})-\widehat{H}_{i}^{k})_{p,q}\right)^{4}\right]\!=\!\frac{d^{2}}{J_{k}^{4}}\sum_{p,q=1}^{d}\mathbb{E}_{k}\!\!\left[\left(\sum_{i=1}^{J_{k}}(\nabla^{2}f_{\sigma_{k}}(x^{k})-\widehat{H}_{i}^{k})_{p,q}\right)^{4}\right]\notag\\
&\overset{\rm(a)}\le \frac{d^{2}}{J_{k}^{4}}\left[\sum_{p,q=1}^{d}\left(3\sum_{i=1}^{J_{k}}(\mathbb{E}_{k}[(\nabla^{2}f_{\sigma_{k}}(x^{k})-\widehat{H}_{i}^{k})_{p,q}^{4}])^{1/2}\right)^{2}\right]\notag\\
&\le\frac{9d^{2}}{J_{k}^{3}}\left[\sum_{p,q=1}^{d}\sum_{i=1}^{J_k}\mathbb{E}_k[(\nabla^{2}f_{\sigma_{k}}(x^{k})-\widehat{H}_i^{k})_{p,q}^{4}]\right]\overset{\rm (b)}=\frac{9d^{2}}{J_{k}^{2}}\left[\sum_{p,q=1}^{d}\mathbb{E}_{u\sim\mathcal{N}(0,I)}[(\nabla^{2}f_{\sigma_{k}}(x^{k})-\widehat{H}^{k})_{p,q}^{4}]\right]\notag\\
&=\frac{9d^{2}}{J_{k}^{2}}\mathbb{E}_{u\sim\mathcal{N}(0,I)}\left[\sum_{p,q=1}^{d}\!(\nabla^{2}f_{\sigma_{k}}(x^{k})-\widehat{H}^{k})_{p,q}^{4}\right]\notag\\
&\le\frac{9d^{2}}{J_{k}^{2}}\mathbb{E}_{u\sim\mathcal{N}(0,I)}\left[\|\nabla^{2}f_{\sigma_{k}}(x^{k})-\widehat{H}^{k}\|_{F}^{4}\right],
\end{align}
where $\widehat{H}_{i}^{k}:=2^{-1}\sigma_{k}^{-2}[f(x^{k}+\sigma_{k} u_{i}^{k})+f(x^{k}-\sigma_{k} u_{i}^{k})-2f(x^{k})](u_{i}^{k}(u_{i}^{k})^{T}-I)$, $\widehat{H}^{k}:=2^{-1}\sigma_{k}^{-2}[f(x^{k}+\sigma_{k} u)+f(x^{k}-\sigma_{k} u)-2f(x^{k})](uu^{T}-I)$ with $u\sim\mathcal{N}(0,I)$, and we used Lemma \ref{4moment} and the fact that $u_i^k$ are i.i.d. samples with $\mathbb{E}_{k}[\widehat H_i^k] = \nabla^{2} f_{\sigma_{k}}(x^{k})$ (see the first equality in \eqref{GSHessian_Lip}) in {\rm(a)}, and we used the fact that  $u_i^k$ are i.i.d. samples in (b). Now, it suffices to derive an upper bound for $\mathbb{E}_{u\sim\mathcal{N}(0,I)}\left[\|\nabla^{2}f_{\sigma_{k}}(x^{k})-\widehat{H}^{k}\|_{F}^{4}\right]$. Notice that
\begin{align}\label{eq2025081103}
&\mathbb{E}_{u\sim\mathcal{N}(0,I)}\left[\|\nabla^{2}f_{\sigma_{k}}(x^{k})-\widehat{H}^{k}\|_{F}^{4}\right]\le 2^{3}\mathbb{E}_{u\sim\mathcal{N}(0,I)}\left[\|\nabla^{2}f_{\sigma_{k}}(x^{k})\|_{F}^{4}+\|\widehat{H}^{k}\|_{F}^{4}\right]\notag\\
&= 2^{3}\mathbb{E}_{u\sim\mathcal{N}(0,I)}\left[\|\widehat{H}^{k}\|_{F}^{4}\right]
+2^{3}\|\nabla^{2}f_{\sigma_{k}}(x^{k})\|_{F}^{4}\overset{\rm (a)}\le2^{4}\mathbb{E}_{u\sim\mathcal{N}(0,I)}\left[\|\widehat{H}^{k}\|_{F}^{4}\right]\notag\\
&\le2^{4}\left[\mathbb{E}_{u\sim\mathcal{N}(0,I)}(|2^{-1}\sigma_{k}^{-2}(f(x^{k}+\sigma_{k} u)+f(x^{k}-\sigma_{k} u)-2f(x^{k}))|^{8})\right]^{1/2}\notag
\\
& \ \ \ \ \ \ \ \times \left[\mathbb{E}_{u\sim\mathcal{N}(0,I)}(\|uu^{T}-I\|_{F}^{8})\right]^{1/2}\notag\\
&\overset{\rm (b)}\le 2^{4}[2^{-8}L^{8}(d+16)^{8}]^{\frac{1}{2}}2^{2}(d+16)^{4}\le 4L^{4}(d+16)^{8},
\end{align}
where (a) holds because the fact that $\mathbb{E}_{u\sim\mathcal{N}(0,I)}[\widehat H^k] = \nabla^{2} f_{\sigma_{k}}(x^{k})$ (see the first equality in \eqref{GSHessian_Lip})
implies that
\begin{align*}
0\le \mathbb{E}_{u\sim\mathcal{N}(0,I)}[\|\nabla^2 f_{\sigma_{k}}(x^{k})-\widehat H^{k}\|_F^{2}] =\mathbb{E}_{u\sim\mathcal{N}(0,I)}(\|\widehat H^{k}\|_F^{2})-\left\|\nabla^2 f_{\sigma_{k}}(x^{k})\right\|_F^{2},
\end{align*}
which in turn implies that 
\[
\|\nabla^{2}f_{\sigma_{k}}(x^{k})\|_{F}^{4} = (\|\nabla^{2}f_{\sigma_{k}}(x^{k})\|_{F}^{2})^2 \le (\mathbb{E}_{u\sim\mathcal{N}(0,I)}(\|\widehat H^{k}\|_F^{2}))^2\le \mathbb{E}_{u\sim\mathcal{N}(0,I)}(\|\widehat H^{k}\|_F^{4}),
\]
while (b) holds in view of \eqref{eq202508250422} and upon applying Lemma~\ref{lem2025102601} with $p = 8$.
 Combining \eqref{eq2025081101}, \eqref{eq2025081102} and \eqref{eq2025081103} with the definition of $J_k$, one can obtain \eqref{thirdexpectation}.
\end{proof}

\subsection{Convergence analysis}\label{sec:alg}

In this subsection, we present our convergence analysis for Algorithm~\ref{algorithm 01}. We impose the following assumptions on the input parameters of Algorithm~\ref{algorithm 01}.
\begin{assumption}\label{Ass_para}
The positive parameters $\rho$, $\alpha$ and $\beta$ in Algorithm~\ref{algorithm 01} satisfy $\rho\in(0,\frac{1}{2})$, $2\rho+\alpha>1$ and $\frac{\rho}{2}+\beta>1$.
\end{assumption}

\begin{lemma}\label{lem042301}
Consider problem \eqref{problem} and suppose that Assumption~\ref{Ass_para} holds. Let $\{x^{k}\}$ be generated by Algorithm \ref{algorithm 01}. Then we have for any $T\ge 1$ that
 \begin{align}\label{bd_diff_xk}
    \sigma_{k_{*}}^{-1}\mathbb{E}\left[\|x^{k_{*}+1}-x^{k_{*}}\|^{3}\right]&=\min_{\lfloor\frac{T}{2}\rfloor\le k\le T}\sigma_{k}^{-1}\mathbb{E}\left[\|x^{k+1}-x^{k}\|^{3}\right]\notag\\
    &\le\frac{72}{L\sqrt{d}T}\left(f_{\gamma}(x^0) - f_*+ D_6\right),
  \end{align}
  and
  \begin{align}
      \E [f_{\sigma_k}(x^k)] - f_{\gamma}(x^0) \le D_6 \ \ \ \forall k\ge 0,\label{bd_f}
  \end{align}
   where $k_{*}:=\max\{l:l\in\arg\min_{\lfloor\frac{T}{2}\rfloor\le  k \le T}\sigma_{k}^{-1}\mathbb{E}\left[\|x^{k+1}-x^{k}\|^{3}\right]\}$,
   \begin{align}\label{defD6}
   \begin{array}{c}
       \displaystyle D_4 := \frac{2^{\frac{\rho}{2}+1}\gamma^{\frac{1}{2}}D_{1}^{\frac{3}{4}}}{L^{\frac{1}{2}}d^{\frac{1}{4}}} \frac{\beta+ \frac{\rho}{2}}{\beta + \frac{\rho}{2}-1}, \quad D_5 := \frac{384\gamma^2D_3}{L^2d}\frac{\alpha+2\rho} {\alpha+2\rho-1},\\
       \displaystyle D_6 := D_4+D_5+ \frac{\gamma^2Ld}{2},
    \end{array}
   \end{align}
   and $D_1$, $D_3$ are defined in Lemma~\ref{lem2025080501}.
\end{lemma}
\begin{proof}
  From Lemma~\ref{lem:cubic_descent}, we have
  \begin{align}
   &f_{\sigma_{k}}(x^{k+1})\le f_{\sigma_{k}}(x^{k})+\langle\nabla f_{\sigma_{k}}(x^{k}),x^{k+1}-x^{k}\rangle+\frac{1}{2}\langle\nabla^{2} f_{\sigma_{k}}(x^{k})(x^{k+1}-x^{k}),x^{k+1}-x^{k}\rangle\notag\\
   &~~~~~~~~~~~~~~~~+\frac{L\sqrt{d}}{6\sigma_{k}}\|x^{k+1}-x^{k}\|^{3}\notag\\
   &\overset{\rm(a)}= f_{\sigma_{k}}(x^{k})+\langle\nabla f_{\sigma_{k}}(x^{k})-v^{k},x^{k+1}-x^{k}\rangle+\frac{1}{2}\langle(\nabla^{2} f_{\sigma_{k}}(x^{k})-H^{k})(x^{k+1}-x^{k}),x^{k+1}-x^{k}\rangle\notag\\
   &~~~~+\frac{L\sqrt{d}}{6\sigma_{k}}[1-(1+\|x^{k}\|)^{3}]\|x^{k+1}-x^{k}\|^{3}+h_{k}(x^{k+1})\notag\\
   &\le f_{\sigma_{k}}(x^{k})+\langle\nabla f_{\sigma_{k}}(x^{k})-v^{k},x^{k+1}-x^{k}\rangle+\frac{1}{2}\langle(\nabla^{2}f_{\sigma_{k}}(x^{k})-H^{k})(x^{k+1}-x^{k}),x^{k+1}-x^{k}\rangle\notag\\
   &~~~~+h_{k}(x^{k+1})\notag\\
   &\le f_{\sigma_{k}}(x^{k})+\|\nabla f_{\sigma_{k}}(x^{k})-v^{k}\|\|x^{k+1}-x^{k}\|+
   \frac{1}{2}\|\nabla^{2}f_{\sigma_{k}}(x^{k})-H^{k}\|\|x^{k+1}-x^{k}\|^{2}\notag\\
   &~~~~+h_{k}(x^{k+1}),\label{eq0513}
  \end{align}
  where {\rm (a)} follows from the definition of $h_{k}$ (see \eqref{defsq}).
  On the other hand, one has
  \begin{align*}
  h_{k}(x^{k+1})&=\langle v^{k},x^{k+1}\!-\!x^{k}\rangle\!+\!\frac{1}{2}\langle H^{k}(x^{k+1}\!-\!x^{k}),x^{k+1}\!-\!x^{k}\rangle\!+\!\frac{L\sqrt{d}}{6\sigma_{k}}(1+\|x^{k}\|)^{3}\|x^{k+1}\!-\!x^{k}\|^{3}\\
  &\overset{\rm(a)}=-\frac{1}{2}\langle H^{k}(x^{k+1}-x^{k}),x^{k+1}-x^{k}\rangle-\frac{L\sqrt{d}}{3\sigma_{k}}(1+\|x^{k}\|)^{3}\|x^{k+1}-x^{k}\|^{3}\\
  &\overset{\rm(b)}\le-\frac{L\sqrt{d}}{12\sigma_{k}}(1+\|x^{k}\|)^{3}\|x^{k+1}-x^{k}\|^{3} ,
  \end{align*}
  where we used \eqref{eq022301} in {\rm(a)} for the equality and {\rm(b)} follows from \eqref{eq022302}.  Combining this with \eqref{eq0513} and \cite[Lemma~2.8(a)]{Starnes2023},\footnote{Note that the $f_{\tau}$ in \cite{Starnes2023} was defined as $f_{\tau}(x) := \frac1{\pi^d}\int_{\R^d}f(x+\sigma u)e^{-\|u\|^2}du$, which corresponds to our $f_{\tau/\sqrt{2}}(x)$. Thus, we have from \cite[Lemma~2.8(a)]{Starnes2023} that $|f_{\sigma_1}(x) - f_{\sigma_2}(x)|\le (\sigma_1^2 - \sigma_2^2)Ld/2$ for all $x\in \R^d$ and $\sigma_1 > \sigma_2 > 0$; see the ArXiv version at https://arxiv.org/pdf/2311.00521 for a proof.} we get 
  \begin{align*}
  &\frac{L\sqrt{d}}{12\sigma_k}(1+\|x^{k}\|)^{3}\|x^{k+1}-x^{k}\|^{3}-\frac{1}{2}(\sigma_{k}^{2}-\sigma_{k+1}^{2})Ld\\
  &\le f_{\sigma_{k}}(x^{k})\!-\!f_{\sigma_{k+1}}(x^{k+1})\!+\!\|\nabla f_{\sigma_{k}}(x^{k})-v^{k}\|\|x^{k+1}\!-\!x^{k}\|\!+\!\frac{1}{2}\|\nabla^{2}f_{\sigma_{k}}(x^{k})-H^{k}\|\|x^{k+1}\!-\!x^{k}\|^{2}\\
  &\le f_{\sigma_{k}}(x^{k})-f_{\sigma_{k+1}}(x^{k+1})+\frac{2}{3}\frac{\|\nabla f_{\sigma_{k}}(x^{k})-v^{k}\|^{\frac{3}{2}}}{\big(L\sqrt{d}/(8\sigma_k)\big)^{\frac{1}{2}}(1+\|x^{k}\|)^{\frac{3}{2}}}+\frac{1}{3}\frac{L\sqrt{d}}{8\sigma_k}(1+\|x^{k}\|)^{3}\|x^{k+1}-x^{k}\|^{3}\\
  &~~+\frac{1}{3}\cdot \left(\frac{1}{2} \right)^3\frac{\|\nabla^{2}f_{\sigma_{k}}(x^{k})-H^{k}\|^{3}}{\left(L\sqrt{d}/(48\sigma_k) \right)^2(1+\|x^{k}\|)^{6}}+\frac{2}{3}\cdot\frac{L\sqrt{d}}{48\sigma_k}(1+\|x^{k}\|)^{3}\|x^{k+1}-x^{k}\|^{3} \\
  &\le f_{\sigma_{k}}(x^{k})\!\!-\!\!f_{\sigma_{k+1}}(x^{k+1})+\frac{2}{3}\frac{\|\nabla f_{\sigma_{k}}(x^{k})-v^{k}\|^{\frac{3}{2}}}{\big(L\sqrt{d}/(8\sigma_k)\big)^{\frac{1}{2}}(1+\|x^{k}\|)^{\frac{3}{2}}}+\frac{1}{3}\frac{L\sqrt{d}}{8\sigma_k}(1+\|x^{k}\|)^{3}\|x^{k+1}-x^{k}\|^{3}\\
  &~~+\frac{96\sigma_k^2\|\nabla^{2}f_{\sigma_{k}}(x^{k})-H^{k}\|^{3}}{L^2d}+\frac{L\sqrt{d}}{72\sigma_k}(1+\|x^{k}\|)^{3}\|x^{k+1}-x^{k}\|^{3} \\
  &\le f_{\sigma_{k}}(x^{k})\!\!-\!\!f_{\sigma_{k+1}}(x^{k+1})+\frac{2\sqrt{\sigma_k}\|\nabla f_{\sigma_{k}}(x^{k})-v^{k}\|^{\frac{3}{2}}}{L^{\frac{1}{2}}d^{\frac{1}{4}}(1+\|x^{k}\|)^{\frac{3}{2}}}+\frac{L\sqrt{d}}{24\sigma_k}(1+\|x^{k}\|)^{3}\|x^{k+1}-x^{k}\|^{3}\\
  &~~+\frac{96\sigma_k^2\|\nabla^{2}f_{\sigma_{k}}(x^{k})-H^{k}\|^{3}}{L^2d}+\frac{L\sqrt{d}}{72\sigma_k}(1+\|x^{k}\|)^{3}\|x^{k+1}-x^{k}\|^{3},
  \end{align*}
  where we used Young's inequality in the second inequality. By rearranging terms in the above inequality and taking expectations on both sides, we can obtain
  \begin{align}\label{eq2025080501}
   &\frac{L\sqrt{d}}{36\sigma_k}(1+\|x^{k}\|)^{3}\mathbb{E}_{k}\left[\|x^{k+1}-x^{k}\|^{3}\right]-\frac{1}{2}(\sigma_{k}^{2}-\sigma_{k+1}^{2})Ld\notag\\
   &\le f_{\sigma_{k}}(x^{k})-\mathbb{E}_{k}[f_{\sigma_{k+1}}(x^{k+1})]+\frac{2\sqrt{\sigma_k}}{L^{\frac{1}{2}}d^{\frac{1}{4}}(1+\|x^{k}\|)^{\frac{3}{2}}}\E_k\left[\|\nabla f_{\sigma_{k}}(x^{k})-v^{k}\|^{\frac{3}{2}} \right]\notag\\
   &~~~~~~~+\frac{96\sigma_k^2}{L^2d}\mathbb{E}_{k}\left[\|\nabla^{2}f_{\sigma_{k}}(x^{k})-H^{k}\|^{3}\right].
  \end{align}
  By Lemma \ref{lem2025080501}, one has
  \begin{align*}
  \mathbb{E}_{k}\left[\|\nabla f_{\sigma_{k}}(x^{k})-v^{k}\|^{\frac{3}{2}}\right]\le\left[\mathbb{E}_{k}\left(\|\nabla f_{\sigma_{k}}(x^{k})-v^{k}\|^{2}\right)\right]^{\frac{3}{4}} \le \frac{D_{1}^{\frac{3}{4}}}{(k+1)^{\beta}}\cdot (1+\|x^{k}\|)^{\frac{3}{2}}
  \end{align*}
  and
  \begin{align*}
  \mathbb{E}_{k}\left[\|\nabla^{2} f_{\sigma_{k}}(x^{k})-H^{k}\|_{F}^{3}\right]\leq \frac{D_3}{(k+1)^{\alpha}}.
  \end{align*}
  These two displays together with \eqref{eq2025080501} and the fact $\|M\|\le \|M\|_F$ for square matrices $M$ imply that
  \begin{align}
      & \frac{L\sqrt{d}}{36\sigma_k}\E_{k} \left[\|x^{k+1} -x^k \|^3\right] \le \frac{L\sqrt{d}}{36\sigma_k}(1+\|x^{k}\|)^{3}\E_{k} \left[\|x^{k+1} -x^k \|^3\right] \nonumber \\
      &\le f_{\sigma_{k}}(x^{k})-\mathbb{E}_{k}[f_{\sigma_{k+1}}(x^{k+1})]+\frac{2\sqrt{\sigma_k}D_{1}^{\frac{3}{4}}}{L^{\frac{1}{2}}d^{\frac{1}{4}}(k+1)^{\beta}}+\frac{96\sigma_k^2D_3}{L^2d(k+1)^{\alpha}} + \frac{1}{2}(\sigma_k^2 -\sigma_{k+1}^2)Ld. \label{sum_diff_x}
  \end{align}
  Taking expectation on both sides of \eqref{sum_diff_x} and summing from $k=0$ to $T\ge 1$, we obtain
  \begin{align}
      &-\frac{\gamma^2 Ld}{2}+\sum_{k=\lfloor \frac{T}{2} \rfloor}^T \frac{L\sqrt{d}}{36\sigma_k}\E \left[\|x^{k+1} -x^k\|^3 \right]
      \le-\frac{\gamma^2 Ld}{2}+\sum_{k=0}^T \frac{L\sqrt{d}}{36\sigma_k}\E \left[\|x^{k+1} -x^k\|^3 \right] \nonumber \\
      &\le f_{\gamma}(x^0) - f_*+ \sum_{k=0}^T\frac{2\sqrt{\sigma_k}D_{1}^{\frac{3}{4}}}{L^{\frac{1}{2}}d^{\frac{1}{4}}(k+1)^{\beta}}+\sum_{k=0}^T\frac{96\sigma_k^2D_3}{L^2d(k+1)^{\alpha}} \nonumber \\
      &\le f_{\gamma}(x^0) - f_*+ \sum_{k=0}^{\infty}\frac{2D_{1}^{\frac{3}{4}}\sqrt{\sigma_k}}{L^{\frac{1}{2}}d^{\frac{1}{4}}(k+1)^{\beta}}+\sum_{k=0}^{\infty}\frac{96\sigma_k^2D_3}{L^2d(k+1)^{\alpha}}\nonumber \\
      &\le f_{\gamma}(x^{0})-f_*+\frac{2D_{1}^{\frac{3}{4}}}{L^{\frac{1}{2}}d^{\frac{1}{4}}}
    \left(\sum_{k=0}^{+\infty}\frac{\sqrt{2^\rho\gamma}}{(k+1)^{\beta+\frac{\rho}{2}}}\right)
    +\frac{96D_3}{L^2d}\left(\sum_{k=0}^{+\infty}\frac{4\gamma^2}{(k+1)^{\alpha+2\rho}}\right) \nonumber \\
    &\le  f_{\gamma}(x^{0})-f_*+\frac{2^{\frac{\rho}{2}+1}\gamma^{\frac{1}{2}}D_{1}^{\frac{3}{4}}}{L^{\frac{1}{2}}d^{\frac{1}{4}}}
    \left(1+ \int_1^{\infty} x^{-\beta -\frac{\rho}{2}} dx\right)
    +\frac{384\gamma^2D_3}{L^2d}\left(1+\int_1^{\infty} x^{-\alpha -2\rho} dx\right) \nonumber \\
    & = f_{\gamma}(x^{0})-f_*+\frac{2^{\frac{\rho}{2}+1}\gamma^{\frac{1}{2}}D_{1}^{\frac{3}{4}}}{L^{\frac{1}{2}}d^{\frac{1}{4}}} \frac{\beta+ \frac{\rho}{2}}{\beta + \frac{\rho}{2}-1}
    +\frac{384\gamma^2D_3}{L^2d}\frac{\alpha+2\rho}{\alpha+2\rho-1}, \label{keysteps}
  \end{align}
  where we used the observation that $\sigma_k = \gamma k^{-\rho} \le 2^\rho\gamma (k+1)^{-\rho} \le 2\gamma (k+1)^{-\rho}$ for all $k \ge 1$ (recall that $\rho\in (0,1/2)$) in the fourth inequality.
  Thus, we have \eqref{bd_diff_xk}. 
  
  Moreover, from \eqref{sum_diff_x}, we obtain that
  \begin{align}
      \E_k[f_{\sigma_{k+1}}(x^{k+1})] - f_{\sigma_k}(x^k) \le \frac{2\sqrt{\sigma_k}D_{1}^{\frac{3}{4}}}{L^{\frac{1}{2}}d^{\frac{1}{4}}(k+1)^{\beta}}+\frac{96\sigma_k^2D_3}{L^2d(k+1)^{\alpha}} + \frac{1}{2}(\sigma_k^2 -\sigma_{k+1}^2)Ld. \nonumber
  \end{align}
  Taking the expectation on both sides of the above inequality and then summing from $k=0$ to $k=T$, we see via a computation analogous to the fourth and fifth inequalities in \eqref{keysteps} that \eqref{bd_f} holds.
\end{proof}

\begin{theorem}\label{thm110401}
Consider problem~\eqref{problem} and suppose that Assumption~\ref{Ass_para} holds. Then the sequence $\{x^k\}$ generated by Algorithm~\ref{algorithm 01} satisfies that, for all $T\ge 2$,
 \begin{align*}
  &\mathbb{E}\bigg[\frac{\|\nabla f_{\sigma_{k_{*}}}(x^{k_{*}})\|}{(1+\|x^{k_*}\|)^3}\bigg]  \\
  &\le\frac{3\gamma^{\frac{1}{3}}L^{\frac{2}{3}}(d+9)^2}{d^{\frac{1}{6}}}\cdot\frac{\left(f_{\gamma}(x^0) - f_*+ D_6\right)^{\frac{1}{3}}}{T^{\frac{1+\rho}{3}}}  + 9L^{\frac{1}{3}}\gamma^{-\frac{1}{3}}d^{\frac{1}{6}}\frac{\left(f_{\gamma}(x^0) - f_*+ D_6\right)^{\frac{2}{3}}}{T^{\frac{1+\rho}{3}}},
  \end{align*}
 where $k_{*}:=\max\{l:l\in\arg\min_{\lfloor\frac{T}{2}\rfloor\le  k \le T}\sigma_{k}^{-1}\mathbb{E}\left[\|x^{k+1}-x^{k}\|^{3}\right]\}$ and $D_6$ is defined in \eqref{defD6}.
\end{theorem}
\begin{proof}
From \eqref{eq022301}, one has
\begin{align*}
\nabla f_{\sigma_{k_*}}(x^{k_*})\!+\!\mathbb{E}_{k_*}[H^{k_*}(x^{k_*+1}\!-\!x^{k_*})]\!+\!\frac{L\sqrt{d}}{2\sigma_{k_*}}(1\!+\!\|x^{k_*}\|)^{3}\mathbb{E}_{k_*}[\|x^{k_*+1}\!-\!x^{k_*}\|(x^{k_*+1}\!-\!x^{k_*})]\!=\!0.
\end{align*}
This implies that
\begin{align}
&\|\nabla f_{\sigma_{k_*}}(x^{k_*})\|\le\mathbb{E}_{k_*}[\|H^{k_*}(x^{k_*+1}-x^{k_*})\|]+\frac{L\sqrt{d}}{2\sigma_{k_*}}(1+\|x^{k_*}\|)^{3} \mathbb{E}_{k_*}[\|x^{k_*+1}-x^{k_*}\|^{2}]\notag\\
&\le[\mathbb{E}_{k_*}(\|H^{k_*}\|^{2})]^{\frac{1}{2}}[\mathbb{E}_{k_*}(\|x^{k_*+1}\!-\!x^{k_*}\|^{2})]^{\frac{1}{2}}\!+\!\frac{L\sqrt{d}}{2\sigma_{k_*}}(1+\|x^{k_*}\|)^{3} 
\mathbb{E}_{k_*}[\|x^{k_*+1}\!-\!x^{k_*}\|^{2}].\label{eq052701}
\end{align}
On the other hand, we can observe that
\begin{align}
 \mathbb{E}_{k_*}[\|H^{k_*}\|^{2}]&=\mathbb{E}_{k_*}\bigg[\bigg\|J_{k_*}^{-1}\sum_{i=1}^{J_{k_*}}\widehat{H}_{i}^{k_*}\bigg\|^{2}\bigg]\le J_{k_*}^{-1}\sum_{i=1}^{J_{k_*}}\mathbb{E}_{k_*}[\|\widehat{H}_{i}^{k_*}\|_{F}^{2}]\le 2^{-2}L^{2}(d+9)^{4},\label{eq052702}
\end{align}
where $J_{k_*}=\lceil(k_*+1)^{\frac{2}{3}\alpha}\rceil$ and $\widehat{H}_{i}^{k}:=2^{-1}\sigma_{k}^{-2}[f(x^{k}+\sigma_{k} u_{i}^{k})+f(x^{k}-\sigma_{k} u_{i}^{k})-2f(x^{k})](u_{i}^{k}(u_{i}^{k})^{T}-I)$, and the last inequality holds in view of the Cauchy-Schwarz inequality, \eqref{eq2025082504} and Lemma~\ref{lem2025102601} with $p = 4$ (see also \eqref{eq2025082502}). Combining \eqref{eq052701} and \eqref{eq052702}, one has
\begin{align*}
  &\|\nabla f_{\sigma_{k_*}}(x^{k_*})\|\le \frac{L}2(d+9)^{2}[\mathbb{E}_{k_*}(\|x^{k_*+1}-x^{k_*}\|^{2})]^{\frac{1}{2}}\!+\!\frac{L\sqrt{d}}{2\sigma_{k_*}}(1+\|x^{k_*}\|)^{3}\mathbb{E}_{k_*}[\|x^{k_*+1}-x^{k_*}\|^{2}]\\
  &\le (1+\|x^{k_*}\|)^{3}\bigg(\frac{L}2(d+9)^{2}[\mathbb{E}_{k_*}(\|x^{k_*+1}-x^{k_*}\|^{2})]^{\frac{1}{2}}\!+\!\frac{L\sqrt{d}}{2\sigma_{k_*}}\mathbb{E}_{k_*}[\|x^{k_*+1}-x^{k_*}\|^{2}]\bigg).
\end{align*}
Dividing both sides of the above inequality by $(1+\|x^{k_*}\|)^{3}$ and taking expectation, we obtain 
\begin{align}
  &\mathbb{E}\bigg[\frac{\|\nabla f_{\sigma_{k_*}}(x^{k_*})\|}{(1+\|x^{k_*}\|)^3}\bigg]\le 2^{-1} L(d+9)^{2}[\mathbb{E}(\|x^{k_*+1}-x^{k_*}\|^{2})]^{1/2}+\frac{L\sqrt{d}}{2\sigma_{k_*}}\mathbb{E}[\|x^{k_*+1}-x^{k_*}\|^{2}] \nonumber \\
  &\le 2^{-1}L(d+9)^2\left(\frac{72\sigma_{k_*}\Delta}{L\sqrt{d}T} \right)^{\frac{1}{3}}  + \frac{L\sqrt{d}}{2\sigma_{k_*}}\left(\frac{72\sigma_{k_*}\Delta}{L\sqrt{d}T} \right)^{\frac{2}{3}} \nonumber \\
  &= 2^{-1} L(d+9)^2\left(\frac{72}{L\sqrt{d}} \right)^{\frac{1}{3}}\Delta^{^\frac{1}{3}}\left(\frac{\sigma_{k_*}}{T}\right)^{\frac{1}{3}}  + \frac{L\sqrt{d}}{2\sigma_{k_*}^{\frac{1}{3}}}\left(\frac{72}{L\sqrt{d}} \right)^{\frac{2}{3}}\Delta^{^\frac{2}{3}} \frac{1}{T^{\frac{2}{3}}} \nonumber \\
  &\le 2^{-1} L(d+9)^2\left(\frac{72}{L\sqrt{d}} \right)^{\frac{1}{3}}\Delta^{^\frac{1}{3}}\left(\frac{\sigma_{\lfloor \frac{T}{2}\rfloor}}{T}\right)^{\frac{1}{3}}  + \frac{L\sqrt{d}}{2\sigma_T^{\frac{1}{3}}}\left(\frac{72}{L\sqrt{d}} \right)^{\frac{2}{3}}\Delta^{^\frac{2}{3}} \frac{1}{T^{\frac{2}{3}}} \nonumber \\
  &\le 2^{-1} L(d+9)^2\left(\frac{72}{L\sqrt{d}} \right)^{\frac{1}{3}}\Delta^{^\frac{1}{3}}\gamma^{\frac{1}{3}}3^{\frac{\rho}{3}}T^{-\frac{1+\rho}{3}}  + \frac{L\sqrt{d}}{2\gamma^{\frac{1}{3}}}\left(\frac{72}{L\sqrt{d}} \right)^{\frac{2}{3}}\Delta^{^\frac{2}{3}} T^{-\frac{2-\rho}{3}} \nonumber \\
  &\le \frac{3\gamma^{\frac{1}{3}}L^{\frac{2}{3}}(d+9)^2}{d^{\frac{1}{6}}}\Delta^{^\frac{1}{3}}T^{-\frac{1+\rho}{3}}  + 9L^{\frac{1}{3}}\gamma^{-\frac{1}{3}}d^{\frac{1}{6}}\Delta^{^\frac{2}{3}} T^{-\frac{2-\rho}{3}} \nonumber \\
  &\le \frac{3\gamma^{\frac{1}{3}}L^{\frac{2}{3}}(d+9)^2}{d^{\frac{1}{6}}}\Delta^{^\frac{1}{3}}T^{-\frac{1+\rho}{3}}  + 9L^{\frac{1}{3}}\gamma^{-\frac{1}{3}}d^{\frac{1}{6}}\Delta^{^\frac{2}{3}} T^{-\frac{1+\rho}{3}}, \nonumber 
\end{align}
where we applied Jensen's inequality to obtain the first summand on the right hand side of the first inequality, the second inequality follows from \eqref{bd_diff_xk} and Jensen's inequality upon letting $\Delta:= f_{\gamma}(x^0) - f_*+ D_6$, the fourth inequality follows from the fact $\lfloor\frac{T}{2}\rfloor\ge\frac{T}{3}$ for any $T\ge2$, while the last two inequalities hold because $\rho \in (0,\frac{1}{2})$.

\end{proof}
\begin{theorem}\label{lem071501}
 Consider problem~\eqref{problem} and suppose that Assumption~\ref{Ass_para} holds. Then the sequence $\{x^k\}$ generated by Algorithm~\ref{algorithm 01} satisfies that, for all $T \ge 1$, 
  \begin{align*}
 &\mathbb{E}\left[\frac{\max\{0,-\lambda_{\min}(\nabla^{2}f_{\sigma_{k_{*}}}(x^{k_{*}}))\}}{(1+\|x^{k_{*}}\|)^{3}}\right]  \le  3L^{\frac{2}{3}}d^{\frac{1}{3}}\gamma^{-\frac{2}{3}}\left( f_{\gamma}(x^0) - f_*+ D_6\right)^{\frac{1}{3}}T^{-\frac{1-2\rho}{3}},
 \end{align*}
where $k_{*}:=\max\{l:l\in\arg\min_{\lfloor\frac{T}{2}\rfloor\le  k \le T}\sigma_{k}^{-1}\mathbb{E}\left[\|x^{k+1}-x^{k}\|^{3}\right]\}$ and $D_6$ is defined in \eqref{defD6}.
\end{theorem}
\begin{proof}
For any $v$ with $\|v\| = 1$, one can deduce from \eqref{eq022302} that for any $k\ge 0$,
\begin{align*}
 \langle v,H^{k}v\rangle+\frac{L\sqrt{d}}{2\sigma_{k}}(1+\|x^{k}\|)^{3}\|x^{k+1}-x^{k}\|\geq0,
\end{align*}
By taking the expectation on both sides of the above display, we can conclude that
\begin{align*}
 \langle v,\nabla^{2}f_{\sigma_{k}}(x^{k}) v\rangle+\frac{L\sqrt{d}}{2\sigma_{k}}(1+\|x^{k}\|)^{3}\mathbb{E}_{k}[\|x^{k+1}-x^{k}\|]\geq0.
\end{align*}
Taking infimum over all $v$ with $\|v\| = 1$ and dividing $(1+\|x^{k}\|)^{3}$ from both sides of the above inequality, we obtain
\begin{align*}
 -\frac{L\sqrt{d}}{2\sigma_{k}}\mathbb{E}_{k}[\|x^{k+1}-x^{k}\|]\le\frac{\min\{0,\lambda_{\mathrm{min}}
 (\nabla^{2}f_{\sigma_{k}}(x^{k}))\}}{(1+\|x^{k}\|)^{3}}.
\end{align*}
By taking the expectation on both sides of the above display, we can obtain that
\begin{align*}
   -\frac{L\sqrt{d}}{2\sigma_{k}}\mathbb{E}[\|x^{k+1}-x^{k}\|]\le
   \mathbb{E}\left[\frac{\min\{0,\lambda_{\mathrm{min}}(\nabla^{2}f_{\sigma_{k}}(x^{k}))\}}{(1+\|x^{k}\|)^{3}}\right].
  \end{align*}
Therefore, we can deduce from the above display that
 \begin{align*}
 &\mathbb{E}\left[\frac{\max\{0,-\lambda_{\min}(\nabla^{2}f_{\sigma_{k_{*}}}(x^{k_{*}}))\}}{(1+\|x^{k_{*}}\|)^{3}}\right]\\
 &\le \frac{L\sqrt{d}}{2\sigma_{k_*}}\mathbb{E}[\|x^{k_{*}+1}-x^{k_{*}}\|]\le  \frac{L\sqrt{d}}{2\sigma_{k_*}}\left(\frac{72\sigma_{k_*}}{L\sqrt{d}T}\left[f_{\gamma}(x^0) - f_*+ D_6\right]\right)^{\frac{1}{3}}\\
 &\le 3L^{\frac{2}{3}}d^{\frac{1}{3}}\left( f_{\gamma}(x^0) - f_*+ D_6\right)^{\frac{1}{3}}\left(\frac{1}{\sigma_{k_{*}}^{2}T}\right)^{\frac{1}{3}}\le 3L^{\frac{2}{3}}d^{\frac{1}{3}}\left( f_{\gamma}(x^0) - f_*+ D_6\right)^{\frac{1}{3}}\left(\frac{1}{\sigma_{T}^{2}T}\right)^{\frac{1}{3}}\\
 &\le3L^{\frac{2}{3}}d^{\frac{1}{3}}\gamma^{-\frac{2}{3}}\left( f_{\gamma}(x^0) - f_*+ D_6\right)^{\frac{1}{3}}T^{-\frac{1-2\rho}{3}},
 \end{align*}
 where the second inequality follows from \eqref{bd_diff_xk} and Jensen's inequality.
\end{proof}

We are now ready to present our complexity result for obtaining $(\epsilon_1,\epsilon_2,\delta)$-second-order stationary points. Our result is based on an assumption on moment bounds; see \eqref{bound_assumption}. We will discuss a sufficient condition for \eqref{bound_assumption} in Proposition~\ref{coro082801} below.

\begin{theorem}[Complexity for $(\epsilon_1,\epsilon_2,\delta)$-second-order stationarity]\label{thm:main_complexity}
   Consider problem~\eqref{problem} and suppose that Assumption~\ref{Ass_para} holds. Let $\{x^k\}$ be the sequence generated by Algorithm~\ref{algorithm 01}. Assume that there exists $\theta >0$ such that
   \begin{equation}\label{bound_assumption}
   \sup_{T\ge 1}\bigg\{\mathbb{E}[\|x^{\hat k}\|^{\theta}]:\; \hat k=\max\{l:l\in\argmin_{\lfloor\frac{T}{2}\rfloor\le  k \le T}\sigma_{k}^{-1}\mathbb{E}\left[\|x^{k+1}-x^{k}\|^{3}\right]\}\bigg\}\le \widetilde M_D
   \end{equation}
 for some $\widetilde M_D\in (0,\infty)$.
Then for any $\delta\in(0,1)$ and any $T \ge 3\left[\frac{\gamma \sqrt{d\pi e}}{\delta (4L)^{-\frac{1}{d}}\min \{5L,1\}^{\frac{1}{d}}}\right]^{\frac{1}{\rho}}+2$, it holds that\footnote{See footnote \ref{footnotehaha} for the well-definedness of the expectation in \eqref{iter_complexity2}.}
  \begin{align}
  & \mathbb{E}\left[\|\nabla f(x^{k_{*}})\|^{\frac{1}{6\lceil 1/\theta\rceil}}\right]
  \le 2\left(1+\sqrt{\widetilde{M}_{D}}\right)\left(\lambda_1 T^{-\rho} + \lambda_2 T^{-\frac{1+\rho}{3}}\right)^{\frac{1}{6\lceil 1/\theta\rceil}},\label{iter_complexity1}\\
 &\mathbb{E}\left[\max\left\{0,-\inf_{\|h\|=1}\suppf(h,\partial_{G,\delta}^{2}f(x^{k_{*}})h)\right\}^{\frac{1}{6\lceil 1/\theta\rceil}}\right]\notag\\
 &~~~~~~~~~~~~~~~~~~~~~~~~~~~
    \le 2\left(1+\sqrt{\widetilde{M}_{D}}\right)\left(\lambda_3T^{-d\rho}
   +\lambda_4T^{-\frac{1-2\rho}{3}}\right)^{\frac{1}{6\lceil 1/\theta\rceil}} \label{iter_complexity2},
  \end{align}
where $k_{*}:=\max\{l:l\in\argmin_{\lfloor\frac{T}{2}\rfloor\le  k \le T}\sigma_{k}^{-1}\mathbb{E}\left[\|x^{k+1}-x^{k}\|^{3}\right]\}$,
\begin{align*}
    & \lambda_1 := 3^\rho\gamma L\sqrt{d}, \\
    &\lambda_2 := \frac{3\gamma^{\frac{1}{3}}L^{\frac{2}{3}}(d+9)^2}{d^{\frac{1}{6}}}\left(f_{\gamma}(x^0) - f_*+ D_6\right)^{\frac{1}{3}}  + 9L^{\frac{1}{3}}\gamma^{-\frac{1}{3}}d^{\frac{1}{6}}\left(f_{\gamma}(x^0) - f_*+ D_6\right)^{\frac{2}{3}}, \\
    &\lambda_3 := 8L\cdot 3^{d\rho}\gamma^d (d\pi e)^{\frac{d}{2}}\delta^{-d}, \\
    &\lambda_4 := 3L^{\frac{2}{3}}d^{\frac{1}{3}}\gamma^{-\frac{2}{3}}\left( f_{\gamma}(x^0) - f_*+ D_6\right)^{\frac{1}{3}},
\end{align*}
and $D_6$ is defined in Lemma~\ref{lem042301}.
\end{theorem}
\begin{proof}
  For each $T\ge 1$ and $\delta > 0$, define 
  \begin{align}
      \epsilon_T := 3^{d\rho}\gamma^d (d\pi e)^{\frac{d}{2}}(\delta^{-d}(4L) )T^{-d\rho}.\label{def_epsilon}
  \end{align}
  Then for any 
  \[
    T \ge 3\left[\frac{\gamma \sqrt{d\pi e}}{\delta (4L)^{-\frac{1}{d}}\min \{5L,1\}^{\frac{1}{d}}}\right]^{\frac{1}{\rho}}+2
  \]
  and $\delta \in (0,1)$, we have $0< \epsilon_T < \min\{5L,1 \}$ and 
  \[
  \sigma_{k_*}\le \sigma_{\lfloor\frac{T}2\rfloor}
   = \gamma (\lfloor T/2\rfloor)^{-\rho} \le \gamma 3^\rho T^{-\rho}= (d\pi e)^{-\frac12}\delta(4L)^{-\frac1d}\epsilon_T^\frac1d,
  \]
  where we used the fact that $\lfloor\frac{T}2\rfloor\ge \frac{T}3$ as $T\ge 2$ for the second inequality, and we used the definition of $\epsilon_T$ for the last equality;
  this means that $\sigma_{k_*}$ satisfies \eqref{2025092201} with $\epsilon_T$ in place of $\epsilon$.
  
 Thus, by Remark~\ref{remark2025092201} and the fact that $2 \le 2(1+ \|x^{k_*}\|)$, we obtain that
  \begin{align*}
    2(1+\|x^{k_{*}}\|)\epsilon_T+[-\lambda_{\min}(\nabla^{2}f_{\sigma_{k_{*}}}(x^{k_{*}}))]\ge -\inf_{\|h\|=1}\suppf(h,\partial_{G,\delta}^{2}f(x^{k_{*}})h).
  \end{align*}
  We can further deduce from the above inequality that
  \begin{align*}
    2(1+\|x^{k_{*}}\|)^{3}\epsilon_T+\max\{0,-\lambda_{\min}(\nabla^{2}f_{\sigma_{k_{*}}}(x^{k_{*}}))\}\ge\max\left\{0, -\inf_{\|h\|=1}\suppf(h,\partial_{G,\delta}^{2}f(x^{k_{*}})h)\right\}.
  \end{align*}
  This together with Jensen's inequality implies that
  \begin{align}
   &\mathbb{E}\left[\frac{\max\{0,-\inf\limits_{\|h\|=1}\suppf(h,\partial_{G,\delta}^{2}f(x^{k_{*}})h)\}^{1/3}}{1+\|x^{k_{*}}\|}\right]  \notag\\
   &\le \left(\mathbb{E}\left[\frac{\max\{0,-\inf\limits_{\|h\|=1}\suppf(h,\partial_{G,\delta}^{2}f(x^{k_{*}})h)\}}{(1+\|x^{k_{*}}\|)^{3}}\right] \right)^{\frac{1}{3}}\notag\\
   &\le \left(2\epsilon_T
   +\mathbb{E}\left[\frac{\max\{0,-\lambda_{\min}(\nabla^{2}f_{\sigma_{k_{*}}}(x^{k_{*}}))\}}{(1+\|x^{k_{*}}\|)^{3}}\right]\right)^{\frac{1}{3}}, \label{bd_second_condition}
  \end{align}
  where the first expectation above is well-defined because $x\mapsto \max\{0,-\inf\limits_{\|h\|=1}\suppf(h,\partial_{G,\delta}^{2}f(x)h)\}$ is lower semicontinuous.\footnote{To see this, recall from \cite[Proposition~2.6.2]{Clarke1983} that $x\mapsto \partial_H^2f(x)$ is upper semicontinuous in the sense of \cite[Proposition~2.6.2(c)]{Clarke1983}, and is locally uniformly bounded. Then it is routine to show that $x\mapsto \partial^2_{G,\delta}f(x)$ is closed in the sense of \cite[Page~34]{Aubin90}. In view of the remark on \cite[Page~42]{Aubin90}, we conclude that $x\mapsto \partial^2_{G,\delta}f(x)$ is upper semicontinuous in the sense of \cite[Definition~1.4.1]{Aubin90}. This together with \cite[Theorem~1.4.16]{Aubin90} gives the upper semicontinuity of $x\mapsto \suppf(h,\partial_{G,\delta}^{2}f(x)h)$. The lower semicontinuity of $x\mapsto \max\{0,-\inf\limits_{\|h\|=1}\suppf(h,\partial_{G,\delta}^{2}f(x)h)\}$ now follows immediately.\label{footnotehaha}} 
Using \eqref{bd_second_condition}, we have upon invoking Lemma~\ref{lem082701}
and recalling the definition of $\widetilde M_D$ that
 \begin{align}
      &\mathbb{E}\left[\max\left\{0,-\inf\limits_{\|h\|=1}\suppf(h,\partial_{G,\delta}^{2}f(x^{k_{*}})h)\right\}^{\frac{1}{6\lceil 1/\theta\rceil}}\right]  \nonumber \\
      &\le 2\left(1+\sqrt{\widetilde{M}_{D}}\right)
      \left(2\epsilon_T
   +\mathbb{E}\left[\frac{\max\{0,-\lambda_{\min}(\nabla^{2}f_{\sigma_{k_{*}}}(x^{k_{*}}))\}}{(1+\|x^{k_{*}}\|)^{3}}\right]\right)^{\frac{1}{6\lceil 1/\theta\rceil}}. \label{bd_support}
  \end{align}
  The relation \eqref{iter_complexity2} now follows immediately from the above display upon invoking Theorem~\ref{lem071501}.
  
  On the other hand, notice that
 \begin{align}
    \|\nabla f(x^{k_{*}})\|&=\|\nabla f(x^{k_{*}})-\nabla f_{\sigma_{k_{*}}}(x^{k_{*}})+\nabla f_{\sigma_{k_{*}}}(x^{k_{*}})\|\notag\\
    &\le\|\nabla f(x^{k_{*}})-\nabla f_{\sigma_{k_{*}}}(x^{k_{*}})\|+\|\nabla f_{\sigma_{k_{*}}}(x^{k_{*}})\|\notag\\
    &=\|\mathbb{E}_{u\sim \mathcal{N}(0,I)}[\nabla f(x^{k_{*}}+\sigma_{k_{*}} u)-\nabla f(x^{k_{*}})]\|+\|\nabla f_{\sigma_{k_{*}}}(x^{k_{*}})\|\notag\\
    &\le\mathbb{E}_{u\sim \mathcal{N}(0,I)}[\|\nabla f(x^{k_{*}}+\sigma_{k_{*}} u)-\nabla f(x^{k_{*}})\|]+\|\nabla f_{\sigma_{k_{*}}}(x^{k_{*}})\|\notag\\
    &\le L\sigma_{k_{*}} \mathbb{E}_{u\sim \mathcal{N}(0,I)}[\|u\|]+\|\nabla f_{\sigma_{k_{*}}}(x^{k_{*}})\|
    \le L\sigma_{k_{*}} \sqrt{d}+\|\nabla f_{\sigma_{k_{*}}}(x^{k_{*}})\|, \label{norm_grad_f}
  \end{align}
  where we used \cite[Lemma 1]{NesterovGS} in the last inequality.
  Therefore, we have 
  \begin{align}
      &\mathbb{E}\left[\frac{\|\nabla f(x^{k_{*}})\|^{1/3}}{1+\|x^{k_{*}}\|}\right] \le \left(\mathbb{E}\left[\frac{\|\nabla f(x^{k_{*}})\|}{(1+\|x^{k_{*}}\|)^3}\right] \right)^{\frac{1}{3}} \le \left( L\sqrt{d}\cdot\sigma_{k_{*}}+\E\left[\frac{\|\nabla f_{\sigma_{k_{*}}}(x^{k_{*}})\|}{(1+\|x^{k_{*}}\|)^{3}}\right] \right)^{\frac{1}{3}} \nonumber \\
      &\le \left(L\sqrt{d}\sigma_{\lfloor\frac{T}{2}\rfloor}+ \E\left[\frac{\|\nabla f_{\sigma_{k_{*}}}(x^{k_{*}})\|}{(1+\|x^{k_{*}}\|)^{3}}\right] \right)^{\frac{1}{3}}\le \left(3^\rho\gamma L\sqrt{d}T^{-\rho }+ \E\left[\frac{\|\nabla f_{\sigma_{k_{*}}}(x^{k_{*}})\|}{(1+\|x^{k_{*}}\|)^{3}}\right] \right)^{\frac{1}{3}}, \label{bd_drad_over_x}
  \end{align}
  where the second inequality holds because of \eqref{norm_grad_f} and the last inequality follows from the fact that $\lfloor\frac{T}{2}\rfloor\ge\frac{T}{3}$, which holds because $T\ge2$. Now, using \eqref{bd_drad_over_x}, we see upon invoking Lemma~\ref{lem082701} and the definition of $\widetilde M_D$ that
  \begin{align}
   &\mathbb{E}\left[\|\nabla f(x^{k_{*}})\|^{\frac{1}{6\lceil 1/\theta\rceil}}\right]\le 2\left(1+\sqrt{\widetilde{M}_{D}}\right)
   \left(3^\rho\gamma L\sqrt{d}T^{-\rho }+ \E\left[\frac{\|\nabla f_{\sigma_{k_{*}}}(x^{k_{*}})\|}{(1+\|x^{k_{*}}\|)^{3}}\right] \right)^{\frac{1}{6\lceil1/\theta\rceil}},\label{eq111103}
  \end{align}
  The relation \eqref{iter_complexity1} now follows immediately from the above display upon invoking Theorem~\ref{thm110401}.                                  \end{proof}

We now provide a sufficient condition for \eqref{bound_assumption}. Notice that condition \eqref{quad_growth2} holds for some $\mu_1 > 0$ and $\mu_2 > 0$ if and only if $\liminf_{\|x\|\to \infty}f(x)/\|x\|^\theta > 0$. 
\begin{proposition}[Sufficient condition for \eqref{bound_assumption}]\label{coro082801}
Consider problem~\eqref{problem} and suppose that Assumption~\ref{Ass_para} holds. Assume that there exist $\theta > 0$, $\mu_1>0$ and $\mu_2>0$ such that 
\begin{align}\label{quad_growth2}
  f(x)\ge \mu_1 \|x\|^\theta -\mu_2 \ \ \forall x\in \R^d.
\end{align}
Then the sequence $\{x^k\}$ generated by Algorithm~\ref{algorithm 01} satisfies that, for all $T\ge 1$,
\begin{align}
\mathbb{E}[\|x^{k_{*}}\|^{\theta}]\le \widetilde M_D,\notag
\end{align}
where $k_{*}:=\max\{l:l\in\argmin_{\lfloor\frac{T}{2}\rfloor\le  k \le T}\sigma_{k}^{-1}\mathbb{E}\left[\|x^{k+1}-x^{k}\|^{3}\right]\}$,
\begin{align*}
\widetilde{M}_{D}&:=\max\{2^{\theta-1},1\}\mu_1^{-1}[\mu_1 (\theta+d)^{\frac{\theta}{2}}+\mu_2+f_\gamma(x^0) + D_6] < \infty,
\end{align*}
and $D_6$ is defined in Lemma~\ref{lem042301}.
\end{proposition}
\begin{proof}
 When $k_* \ge 1$, we can obtain from Lemma \ref{lem042301} that
\begin{align}\label{equationhaha}
&\mathbb{E}[f_{\sigma_{k_{*}}}(x^{k_*})]\le f_\gamma(x^0) + D_6.
\end{align}
Notice that the above inequality also holds when $k_* = 0$.

On the other hand, according to \eqref{quad_growth2}, we have
\begin{align*}
  &f_{\sigma_{k_{*}}}(x^{k_{*}})=\mathbb{E}_{u\sim\mathcal{N}(0,I)}[f(x^{k_{*}}+\sigma_{k_{*}} u)]\ge\mu_1\mathbb{E}_{u\sim\mathcal{N}(0,I)}[\|x^{k_*} +\sigma_{k_*}u\|^{\theta}] -\mu_2\\
  &=\frac{\mu_1}{w_{\theta}}\mathbb{E}_{u\sim\mathcal{N}(0,I)}[w_{\theta}\|x^{k_{*}}+\sigma_{k_{*}} u\|^{\theta}+w_{\theta}\|\sigma_{k_{*}} u\|^{\theta}-w_{\theta}\|\sigma_{k_{*}} u\|^{\theta}]-\mu_2\\
  &\ge\frac{\mu_1}{w_{\theta}}\|x^{k_*}\|^{\theta}-\mu_1 (\theta+d)^{\frac{\theta}{2}}\sigma_{k_{*}}^{\theta}-\mu_2,
\end{align*}
where $w_{\theta}:=\max\{2^{\theta-1},1\}$ and we use the inequality $\|x^{k_{*}}\|^{\theta}\le(\|x^{k_{*}}+\sigma_{k_{*}} u\|+\|\sigma_{k_{*}} u\|)^{\theta}$,  \cite[Exercise 17(a) in Chapter 3]{Rudin1987} and \cite[Lemma 1]{NesterovGS} for the last inequality. This together with $\sigma_k \in (0,1)$ implies that
\begin{align}
 \mathbb{E}[f_{\sigma_{k_{*}}}(x^{k_{*}})] \ge\frac{\mu_1}{w_{\theta}}\mathbb{E}[\|x^{k_*}\|^{\theta}]-\mu_1 (\theta+d)^{\frac{\theta}{2}}-\mu_2.\label{eq082802}
\end{align}
Combining \eqref{equationhaha} and \eqref{eq082802}, one has
\begin{align*}
 \frac{\mu_1}{w_{\theta}}\mathbb{E}[\|x^{k_*}\|^{\theta}]&\le \mu_1 (\theta+d)^{\frac{\theta}{2}}+\mu_2+f_\gamma(x^0) + D_6.
\end{align*}
The desired conclusion now follows immediately from the above display and the definition of $\widetilde M_D$.
\end{proof}